\documentclass[12pt]{article}

\usepackage{amsmath,amssymb,amsfonts}
\usepackage{amsthm}
\usepackage[margin=1in]{geometry}

\usepackage{graphicx}
\usepackage{booktabs}

\usepackage{hyperref}
\usepackage{microtype}
\newtheorem{theorem}{Theorem}
\newtheorem{corollary}{Corollary}
\newtheorem{proposition}{Proposition}

\newtheorem{remark}{Remark}
\newtheorem{definition}{Definition}

\newcommand{\fplus}{f^{+}}

\newcommand{\fpm}{f^{\pm}}

\begin{document}

\title{State-Feedback Control of Logistic-Based Gene Regulatory
       Networks: Closed-Form Lyapunov Certificates, Monostabilization,
       and Delay-Uniform Stability}

\author{Ismail Belgacem%
  \thanks{Independent researcher, Mezaourou Ghazaouet, Tlemcen 13421, Algeria.\newline
           Email: \texttt{ismail.belgacem.81@gmail.com}}}

\date{June 2026}

\maketitle

\begin{abstract}
Gene regulatory networks (GRNs) are high-value targets for therapeutic and
synthetic-biology intervention, yet the classical Hill-function models used to
describe them carry a structural defect that is critical for control: the
production term vanishes when the activator is absent, causing loss of
controllability under multiplicative actuation and network decoupling under
additive actuation---precisely the low-expression regime in which cells most
often operate. Building on companion works that establish logistic functions as
robust Hill alternatives~\cite{belgacem2025exploring,belgacem2026logistic}, we develop
an \emph{additive state-feedback} framework for logistic-based GRNs in which a
feedforward-plus-proportional law renders any positive setpoint a closed-loop
equilibrium, whether or not it is an equilibrium of the uncontrolled dynamics.
We prove local exponential stability under an explicit Gershgorin gain bound and
\emph{global} exponential stability of the nonlinear closed loop---via a common
quadratic Lyapunov function built on the logistic sector bound---under the gain
condition $(\gamma_1+K_1)(\gamma_2+K_2)>\kappa_1\kappa_2\lambda^2/64$, with a
closed-form rate. A diagonal Lyapunov certificate $P=\operatorname{diag}(B,A)$
for the sign-structured closed-loop Jacobian supplies explicit settling-time and
input-to-state-stability bounds, the rate $\min_i(\gamma_i+K_i)$ being
arbitrarily accelerable by gain. Two complementary scalar results follow from the
same universal Lipschitz bound $\lambda/4$ of the logistic derivative: a
\emph{parameter-uniform monostabilization budget} $K^{*}=\kappa\lambda/4-\gamma$
(Theorem~\ref{thm:monostab}) for bistable self-activation switches, and a
\emph{Halanay-type delay-uniform global exponential stability} theorem
(Theorem~\ref{thm:halanay}) under $\gamma+K>\kappa\lambda/4$, with closed-form
two-sided bounds on the delay-dependent rate. A worked T-cell switch example and
side-by-side comparisons with the Hill counterpart---whose off-diagonal Jacobian
coupling decays as $\Theta(x_{d,1}^{n-1})\to 0$ as the activator setpoint
$x_{d,1}\to 0$, whereas the logistic coupling converges to the strictly positive
limit $\kappa\lambda f^+(0)(1-f^+(0))$---illustrate the results.
\end{abstract}


\section{Introduction}
\label{sec:intro}

The control of gene regulatory networks (GRNs) is a central challenge in systems
and synthetic biology, with applications spanning metabolic engineering, drug
discovery, and gene therapy~\cite{alon2006introduction,dejong2002modeling}. Precise
regulation of gene expression requires both faithful mathematical models of the
underlying molecular interactions and rigorous control-theoretic frameworks capable
of driving network states to desired targets despite biological noise and parameter
uncertainty.

Hill functions, originally derived to model cooperative ligand binding to
hemoglobin~\cite{hill1910possible,weiss1997hill}, have become the standard sigmoid
building blocks for GRN models~\cite{dejong2002modeling,alon2006introduction}. Their
appeal lies in a compact parametric form: $h^+(x,\theta,n) = x^n/(x^n+\theta^n)$
interpolates smoothly between zero and one as the regulator concentration $x$ passes
through the threshold $\theta$, with cooperativity controlled by the Hill coefficient
$n$. However, as we established in our companion
papers~\cite{belgacem2025exploring,belgacem2026logistic}, Hill functions carry a
structural limitation that is particularly detrimental for control design:
$h^+(0,\theta,n) = 0$ identically. When the activating regulator is absent, the
production term vanishes regardless of any control input, causing complete loss of
controllability precisely in the low-expression regimes most common in biological
practice---OFF states, post-perturbation recovery, and transitions between attractors.

Logistic functions $f^+(x,\theta,\lambda) = 1/(1+e^{-\lambda(x-\theta)})$ provide a
principled, biologically grounded alternative~\cite{belgacem2025exploring,belgacem2026logistic,belgacem2026beyond}. Their sigmoid shape closely approximates
Hill behavior for large steepness $\lambda$ while maintaining a strictly positive
output at all finite concentrations, including $x = 0$:
$f^+(0,\theta,\lambda) = 1/(1+e^{\lambda\theta}) > 0$. This non-zero basal response
is not an artifact but a faithful reflection of biological reality: promoter leakiness
sustains measurable transcription even in nominally repressed genes~\cite{huang2015effects},
and long-lived low-expression states are experimentally well-documented in systems
ranging from the yeast GAL network~\cite{acar2005enhancement} to bacterial
operons~\cite{jacob1961genetic}. Beyond biological fidelity, the smooth, bounded, and
analytically invertible nature of the logistic function makes it substantially more
tractable for control design than the power-law Hill function.

The present paper develops the additive state-feedback (Architecture~B)
branch of this control framework, together with two complementary scalar
control results that share the same Lyapunov-based reasoning. It is part
of an ongoing series: Paper~I~\cite{belgacem2025exploring} introduced
logistic functions as robust Hill alternatives; Paper~II~\cite{belgacem2026logistic}
established prevention of expression shutdown and numerical stability;
a companion paper~\cite{belgacem2026beyond} addresses delay-coupled
networks at the general theoretical level, and a sister
paper~\cite{belgacem2026sustained} establishes a delay-driven Hopf
bifurcation in the two-gene logistic oscillator; a companion
paper~\cite{belgacem2026extensions} extends the
framework to immunology, hematopoiesis, and stochastic dynamics (and
supplies the closed-form bistability theorem used in the
monostabilization result of Section~\ref{sec:monostab}). None of these prior papers addresses
the additive state-feedback design and its closed-form Lyapunov
certificate; that is the contribution of the present work.

The present analysis also builds on the author's prior research
programme on stability analysis, reduction, and control of biological
and biochemical
networks~\cite{belgacem2018reduction,belgacem2014cdc,belgacem2014mtns,belgacem2013cab,belgacem2013cdc,belgacem2013acta,belgacem2012ifac,belgacem2014med},
and on the computational analysis of high-dimensional gene network
dynamics, including chaos and bifurcations in ring
circuits~\cite{belgacem2025glass,farcot2019chaos}.

\medskip
\noindent\textbf{Contributions.} The specific contributions of the
present paper are:
\begin{enumerate}
  \item \emph{Additive state-feedback control, local and global.} A
        feedforward-plus-proportional control law
        (Equation~\eqref{eq:sf_law}) that makes any positive setpoint a
        closed-loop equilibrium of the additive-input architecture,
        regardless of whether it is an equilibrium of the uncontrolled
        dynamics. We establish local exponential stability under a
        sufficient Gershgorin gain bound (Theorem~\ref{thm:sf_stability})
        and, via a common quadratic Lyapunov function built on the
        logistic sector bound, \emph{global} exponential stability of the
        nonlinear closed loop under the explicit gain condition
        $(\gamma_1+K_1)(\gamma_2+K_2)>\kappa_1\kappa_2\lambda^2/64$, with
        a closed-form rate and global settling-time bound
        (Theorem~\ref{thm:global_sf}, Corollary~\ref{cor:global_settling}).
        The same Lyapunov function shows the closed loop tracks
        bounded-velocity time-varying references with a proportionally
        bounded, gain-attenuable error (Corollary~\ref{cor:tracking}).
  \item \emph{Logistic structural advantage in additive control.} A formal
        comparison (Proposition~\ref{prop:sf_advantage}) showing that Hill
        activation makes both the regulatory production and the off-diagonal
        Jacobian coupling vanish as the activator setpoint approaches zero
        (network decoupling in additive control), and in the
        output-multiplicative architecture causes outright loss of
        controllability, whereas the
        logistic regulatory production and Jacobian coupling remain strictly
        positive throughout the entire positive orthant. We quantify this
        advantage numerically in Section~\ref{sec:numerical}.
  \item \emph{Closed-form Lyapunov function and quantitative performance
        bounds.} The closed-loop Jacobian belongs to a sign-structured class
        of $2\times 2$ matrices admitting an explicit diagonal Lyapunov
        weighting $P=\operatorname{diag}(B,A)$
        (Proposition~\ref{prop:lyapunov}); this yields an explicit
        settling-time bound that scales as $\rho^{-1}\ln(1/\delta)$ with
        $\rho=\min(\gamma_i+K_i)$ (Corollary~\ref{cor:settling}), matching
        direct ODE simulation to within $1\%$--$16\%$ depending on
        initial-condition direction, and an ISS-type ultimate-bound
        certificate scaling as $D/\rho$ under bounded disturbances
        $\|\boldsymbol{\eta}\|_2\leq D$ (Proposition~\ref{prop:iss}). The
        convergence rate $\rho=\min(\gamma_i+K_i)$ is arbitrarily
        accelerable by gain (Remark~\ref{rem:lyapunov_sf}), and the
        disturbance ultimate bound shrinks correspondingly. The same diagonal
        weighting certifies \emph{any} closed loop sharing this sign
        structure, so the Lyapunov certificate is architecture-independent.
  \item \emph{Monostabilization of bistable switches and delay-uniform
        global stability.} For the scalar logistic self-activation switch
        in its bistable regime, the
        \emph{parameter-uniform monostabilization budget}
        $K^{*}=\kappa\lambda/4-\gamma$ (Theorem~\ref{thm:monostab}) is the
        minimum proportional gain above which the closed loop is
        monostable for every set-point and threshold; and for the scalar
        delayed feedback loop, the Halanay-type condition
        $\gamma+K>\kappa\lambda/4$ (Theorem~\ref{thm:halanay}) yields
        \emph{delay-uniform} global exponential stability, with
        closed-form two-sided bounds on the delay-dependent rate
        (Corollary~\ref{cor:halanay_rate}). The global Lipschitz bound
        $\lambda/4$ of the logistic derivative underpins not only these
        scalar results but also the global 2-D theorem
        (Theorem~\ref{thm:global_sf}), giving the paper a single
        structural backbone. A worked
        T-cell switch example
        (Section~\ref{sec:control_example}, Figure~\ref{fig:control})
        illustrates both bounds.
  \item \emph{Numerical comparison.} A side-by-side benchmark
        (Section~\ref{sec:numerical}) using identical setpoints, gains, and
        parameters validates the Lyapunov-bound prediction and exposes the
        structural decoupling of Hill state-feedback as the activator
        setpoint approaches zero, in agreement with
        Proposition~\ref{prop:sf_advantage}.
\end{enumerate}

\medskip
\noindent\textbf{Outline.} Section~\ref{sec:model} introduces the
logistic gene regulatory model and the two control architectures.
Section~\ref{sec:state_feedback} presents the additive state-feedback
design with the local Gershgorin stability theorem, a global
exponential stability theorem via a common quadratic Lyapunov function,
and the logistic structural advantage. Section~\ref{sec:lyapunov} develops the
closed-form Lyapunov framework that supplies an explicit settling-time
bound and an ISS certificate for the closed loop.
Section~\ref{sec:monostab} establishes the parameter-uniform
monostabilization budget $K^*=\kappa\lambda/4-\gamma$ for the scalar
self-activation switch. Section~\ref{sec:halanay} gives the
Halanay-type delay-uniform global exponential stability theorem.
Section~\ref{sec:control_example} works out the T-cell switch example.
Section~\ref{sec:numerical} presents numerical simulations.
Section~\ref{sec:conclusion} concludes with future directions.

\section{Logistic Gene Regulatory Network Model and Control Architecture}
\label{sec:model}

\begin{definition}[Logistic sigmoid functions]
\label{def:logistic}
For $x\in\mathbb{R}$, threshold $\theta>0$, and steepness $\lambda>0$, define the
\emph{activating} and \emph{repressing} logistic sigmoids:
\begin{equation}
\begin{aligned}
f^+(x;\theta,\lambda) &= \frac{1}{1+e^{-\lambda(x-\theta)}}, \\
f^-(x;\theta,\lambda) &= \frac{1}{1+e^{\lambda(x-\theta)}} = 1 - f^+(x;\theta,\lambda).
\end{aligned}
\label{eq:logistic_def}
\end{equation}
Both functions map $\mathbb{R}$ into $(0,1)$ and satisfy
$f^+(0;\theta,\lambda) = 1/(1+e^{\lambda\theta}) > 0$ and
$f^-(0;\theta,\lambda) = 1/(1+e^{-\lambda\theta}) < 1$.
Their derivatives satisfy $\tfrac{d}{dx}f^\pm = \pm\lambda f^\pm(1-f^\pm)$,
with the universal bound $f(1-f)\leq 1/4$, attained at the threshold $x=\theta$.
\end{definition}

Controlling biological networks is a frontier of substantial practical promise.
Our prior qualitative-control work on piecewise-affine and Hill-based
genetic networks~\cite{belgacem2020control,chambon2020qualitative}
demonstrated that sliding-mode and qualitative feedback designs offer
rigorous frameworks for stabilising
gene expression. Optogenetic platforms such as the
Diya light-illumination system~\cite{kumar2023diya} are now mature enough to
implement closed-loop feedback at the single-cell or population level.
The control-theoretic difficulty with Hill functions is structural rather
than incidental: the activating factor
$h^+(x,\theta,n) = x^n/(x^n+\theta^n)$ vanishes identically at $x=0$,
so when an activating regulator is absent, the production term is zero
\emph{independently of the control input}, and the system is uncontrollable
in the activator-absent regime. The logistic function
$f^+(x,\theta,\lambda)$ of Definition~\ref{def:logistic} satisfies
$f^+(0;\theta,\lambda) = 1/(1+e^{\lambda\theta}) > 0$, so the production
term never vanishes and the system remains controllable throughout the
strictly positive orthant. This is the structural property that drives every
comparison result in the present paper~\cite{belgacem2025exploring,belgacem2026logistic}.

\subsection{Biological Foundations: Persistent Basal Expression}
\label{sec:bio}

The non-zero basal output of the logistic function is not a mathematical
convenience: it is a faithful reflection of biological reality. In
activator-regulated systems, transcription proceeds at leaky basal rates even
when the activator is absent, and these low-expression states persist as long
as the activator is absent. The yeast GAL regulatory network is the canonical
example: cells maintain the OFF state with measurable expression across
multiple cell divisions without galactose~\cite{acar2005enhancement}.
Quantitative single-molecule studies have shown that promoter leakage
stabilises low-expression states in auto-regulatory circuits while reducing
expression noise~\cite{huang2015effects}. Long lifetimes of the basal
attractor follow from the rare-event statistics of binding/unbinding kinetics
in genetic switches~\cite{walczak2005absolute}.

Repressor-free states are equally persistent. In bacterial operons such as
the \emph{lac} system, knockout of the \emph{lacI} repressor yields
constitutive expression of \emph{lacZYA} for as long as the cell survives,
without inducer~\cite{jacob1961genetic}. In phage~$\lambda$ lysogeny,
inactivation of the CI repressor switches the cell to the lytic programme
and dominates dynamics until lysis~\cite{ackers1982quantitative}.

The logistic framework captures both regimes intrinsically: the activator
basal $f^+(0;\theta,\lambda) = (1+e^{\lambda\theta})^{-1}$ is positive but
small (engineerable via $\lambda\theta$), and the repressor-absent value
$f^-(0;\theta,\lambda) = (1+e^{-\lambda\theta})^{-1}$ is close to but less
than~$1$. Hill functions force $h^+(0) = 0$ and $h^-(0) = 1$ identically,
which is incompatible with the leaky biology and---more importantly for the
present paper---incompatible with controllability at zero
input~\cite{belgacem2025exploring,belgacem2026logistic}.

\subsection{Controlled Logistic-Based Network Model}
\label{sec:logistic-model}

A controlled logistic-based GRN can take several mathematical forms depending
on how the external input couples to the regulatory machinery. Two physically
motivated architectures are used in this paper, both preserving the
controllability property of the logistic framework:

\smallskip
\noindent\textbf{Architecture A (argument-modulating).}\;
The control input $u_{ij}\geq 0$ scales the regulator concentration entering
the logistic argument, modelling situations where the input modulates the
\emph{effective} concentration of a regulator (e.g.\ optogenetic switching
that activates a fraction $u_{ij}$ of regulator molecules):
\begin{equation}
\dot{x}_i \;=\; \kappa_i\, f_i\bigl(\mathbf{x};\mathbf{u}_i\bigr) - \gamma_i x_i,
\qquad i=1,\ldots,n,
\label{eq:control_expression}
\end{equation}
where each logistic factor in $f_i$ has the form
$1/\bigl(1+e^{\mp\lambda(u_{ij}x_j-\theta_{ij})}\bigr)$.

\smallskip
\noindent\textbf{Architecture B (additive).}\;
The control input $u_i\in\mathbb{R}$ is added directly to the production rate,
modelling externally supplied transcript or protein:
\begin{equation}
\dot{x}_i \;=\; \kappa_i\, f_i(\mathbf{x}) - \gamma_i x_i + u_i,
\qquad i=1,\ldots,n,
\label{eq:additive_control}
\end{equation}
where $f_i$ is a product of plain logistic factors of $x_j$ alone.

\smallskip
In both architectures $\kappa_i>0$ is the maximal production rate, $\gamma_i>0$
the degradation rate, and $f_i\in (0,1]$ is a product of logistic factors
$f^\pm$ encoding the regulatory logic. For one activator
$j\in\mathcal{A}_i$ and one repressor $k\in\mathcal{R}_i$ in Architecture~B:
\begin{equation}
f_i(\mathbf{x}) \;=\;
\frac{1}{1+e^{-\lambda(x_j-\theta_{ij})}}\cdot
\frac{1}{1+e^{\lambda(x_k-\theta_{ik})}}.
\label{eq:fi-AR}
\end{equation}
At $x_j=x_k=0$:
$f_i(\mathbf 0) = (1+e^{\lambda\theta_{ij}})^{-1}(1+e^{-\lambda\theta_{ik}})^{-1}>0$,
so the logistic regulatory production is strictly positive. The Hill counterpart
$x_j^n/(x_j^n+\theta_{ij}^n)\cdot \theta_{ik}^n/(x_k^n+\theta_{ik}^n)$ vanishes
at $x_j=0$ (the activator-absent regime), so the regulatory contribution
disappears entirely. The consequences differ by architecture. In
Architecture~B the vanishing of $h_i$ leaves the additive control alone to
supply the entire production flux of gene~$i$. A sharper effect arises in the
\emph{output-multiplicative} variant
$\dot{x}_i=\kappa_i u_i f_i(\mathbf{x})-\gamma_i x_i$, in which the input
enters as a multiplicative gain on the regulatory output: there
$h_i(\mathbf{x})=0$ annihilates all control authority over $\dot{x}_i$---an
outright loss of controllability---whereas the logistic $f_i^{\mathrm{log}}>0$
retains it everywhere on the closed positive orthant. In Architecture~A, by
contrast, the input acts through the product $u_{ij}x_j$, which vanishes at
$x_j=0$ \emph{independently of the regulatory function}: both the Hill and the
logistic models then lose controllability through regulator~$j$ at the
boundary $x_j=0$, and Architecture~A is controllable only on the open orthant
$x_j>0$. For an $n$-dimensional network with multiple parallel activators and
repressors, $f_i$ generalises to
$f_i(\mathbf{x}) =
\prod_{j\in\mathcal{A}_i} f^+(x_j;\theta_{ij},\lambda)
\cdot
\prod_{k\in\mathcal{R}_i} f^-(x_k;\theta_{ik},\lambda)$,
which remains strictly positive on $\mathbb{R}_{\geq 0}^n$ for any choice of
control architecture.

\smallskip
\noindent\textbf{Choice of architecture.}\;
The two architectures call for different controller designs.
Architecture~A, in which the control enters nonlinearly through the
logistic argument, is naturally suited to sliding-mode control: the equivalent
control can be solved in closed form by inverting the logistic, and
the boundary-layer linearisation yields a clean stability bound.
Architecture~B, in which the control enters affinely, is naturally
suited to state-feedback: a feedforward-plus-proportional law makes
any positive setpoint a closed-loop equilibrium, and Gershgorin's
theorem yields an explicit gain condition. The present paper develops
the Architecture~B branch in detail.

\smallskip
\noindent\textbf{Two-gene benchmark.}\;
Throughout the rest of this paper we use the following two-gene
oscillator in Architecture~B: gene~1 is repressed by gene~2, gene~2
is activated by gene~1, and each gene carries an additive control input
$u_i$:
\begin{equation}
\boxed{
\begin{aligned}
\dot{x}_1 &= \frac{\kappa_1}{1+e^{\lambda(x_2 - \theta_2)}} - \gamma_1 x_1 + u_1,\\[2pt]
\dot{x}_2 &= \frac{\kappa_2}{1+e^{-\lambda(x_1 - \theta_1)}} - \gamma_2 x_2 + u_2.
\end{aligned}
}
\label{eq:two_gene_B}
\end{equation}
The uncontrolled limit $u_1=u_2=0$ is the classical activator--repressor
oscillator. At $\mathbf{x}=\mathbf{0}$ the regulatory productions evaluate
to $\kappa_1/(1+e^{-\lambda\theta_2})$ and $\kappa_2/(1+e^{\lambda\theta_1})$,
both strictly positive, so the system has non-zero
``wake-up'' drive even from a deep OFF state, regardless of $\mathbf{u}$.

\section{State-Feedback Control Strategies}
\label{sec:state_feedback}

For Architecture~B (additive control,~\eqref{eq:additive_control}), we
develop a feedforward-plus-proportional-feedback design that makes any
positive setpoint $\mathbf{x}_d$ a closed-loop equilibrium---not necessarily
an equilibrium of the uncontrolled system. The control law is
\begin{equation}
u_i \;=\; \underbrace{\gamma_i x_{d,i} - \kappa_i f_i(\mathbf{x}_d)}_{u_i^{\mathrm{ff}}}
   \;-\; K_i (x_i - x_{d,i}),\qquad K_i>0.
\label{eq:sf_law}
\end{equation}
The feedforward term $u_i^{\mathrm{ff}}$ exactly cancels the steady-state
mismatch $\gamma_i x_{d,i} - \kappa_i f_i(\mathbf{x}_d)$, while the feedback
term provides linear restoring action. Direct substitution
into~\eqref{eq:additive_control} confirms that $\mathbf{x}_d$ is an
equilibrium of the closed-loop system for every positive setpoint, removing
the implicit requirement (made in earlier formulations) that $\mathbf{x}_d$
already be an equilibrium of the uncontrolled dynamics.

Linearizing around $\mathbf{x}_d$ with $\mathbf{e}=\mathbf{x}-\mathbf{x}_d$:
\begin{equation}
\dot{\mathbf{e}} \;=\; \bigl(J_f - \Gamma - K\bigr)\mathbf{e}+O(\|\mathbf{e}\|^2),
\label{eq:linearized_cl}
\end{equation}
where $[J_f]_{ij}=\kappa_i \,(\partial f_i/\partial x_j)|_{\mathbf{x}_d}$,
$\Gamma=\mathrm{diag}(\gamma_i)$, and $K=\mathrm{diag}(K_i)$.

\subsection{Local exponential stability: a Gershgorin gain bound}
\label{sec:sf_local}

\begin{theorem}[State-Feedback Exponential Stability]
\label{thm:sf_stability}
For the additive system~\eqref{eq:additive_control} under the law~\eqref{eq:sf_law},
$\mathbf{x}_d$ is an equilibrium for any $K_i\geq 0$. A sufficient condition for
exponential stability of the linearization is the Gershgorin bound:
\begin{equation}
K_i \;>\; [J_f]_{ii}+\sum_{j\neq i}\bigl|[J_f]_{ij}\bigr|-\gamma_i,
\quad\forall i,
\label{eq:gershgorin_K}
\end{equation}
in which case every eigenvalue of $J_f-\Gamma-K$ lies in the open left half-plane.
In particular, for the two-gene logistic activator--repressor network without
self-regulation, $[J_f]_{ii}=0$ and the condition reduces to
$K_i>|[J_f]_{ij}|-\gamma_i$ for the single off-diagonal entry $j\neq i$, which is
always achievable for sufficiently large $K_i$.
\end{theorem}

\begin{proof}
Substituting~\eqref{eq:sf_law} at $\mathbf{x}=\mathbf{x}_d$ gives
$\dot{x}_i\big|_{\mathbf{x}_d}=\kappa_i f_i(\mathbf{x}_d)-\gamma_i x_{d,i}
+\gamma_i x_{d,i}-\kappa_i f_i(\mathbf{x}_d)=0$, so $\mathbf{x}_d$ is an equilibrium.
The linearization~\eqref{eq:linearized_cl} follows directly. By the Gershgorin
disc theorem, every eigenvalue of $J_f-\Gamma-K$ lies in some disc centered at
$[J_f]_{ii}-\gamma_i-K_i$ with radius $\sum_{j\neq i}|[J_f]_{ij}|$.
Condition~\eqref{eq:gershgorin_K} places every disc strictly in the open left
half-plane, ensuring exponential stability.
\end{proof}

\begin{remark}[Explicit Gershgorin bounds for the two-gene oscillator]
\label{rem:gershgorin_numerical}
For the two-gene activator--repressor network at setpoint $(x_{d,1},x_{d,2})=(0.8,0.6)$
with $\kappa_i=1$, $\gamma_i=0.5$, $\lambda=5$, $\theta_i=0.5$ and no
self-regulation ($[J_f]_{ii}=0$), the off-diagonal entries are:
\begin{align*}
|[J_f]_{12}| &= \kappa_1\lambda\,f^-(1-f^-)\big|_{x_{d,2}}\\
  &\approx 5\times 0.378\times 0.622 \approx 1.175,\\
|[J_f]_{21}| &= \kappa_2\lambda\,f^+(1-f^+)\big|_{x_{d,1}}\\
  &\approx 5\times 0.818\times 0.182 \approx 0.746.
\end{align*}
The Gershgorin condition~\eqref{eq:gershgorin_K} therefore requires
$K_1 > 1.175 - 0.5 = 0.675$ and $K_2 > 0.746 - 0.5 = 0.246$
for exponential stability of the linearisation. The corresponding
closed-loop convergence rate is $\rho = \min(\gamma_i+K_i)$
(Remark~\ref{rem:lyapunov_sf} of Section~\ref{sec:lyapunov}): e.g.\
the minimum stabilising gains $K_1=0.68$, $K_2=0.25$ yield
$\rho \approx 0.75$, while the moderately larger gains $K_1=K_2=1$
used in Section~\ref{sec:numerical} yield $\rho=1.5$, a twofold
increase obtained at the cost of a correspondingly larger
feedforward-plus-feedback control amplitude. The disturbance ultimate
bound of Proposition~\ref{prop:iss} shrinks correspondingly.
For comparison, the Hill activator coupling at the same setpoint gives
$|[J_f^{\mathrm{Hill}}]_{21}| = \kappa_2 n\theta_1^n (x_{d,1})^{n-1}/(\theta_1^n+(x_{d,1})^n)^2\approx 0.505$
(weaker coupling at this setpoint), but this entry vanishes as
$x_{d,1}\to 0$, leaving only the diagonal gain $-\gamma_i-K_i$ to stabilise
the network---the exact loss of coupling formalised in
Proposition~\ref{prop:sf_advantage}.
\end{remark}

\subsection{Global exponential stability via a common quadratic Lyapunov function}
\label{sec:sf_global}

Theorem~\ref{thm:sf_stability} certifies exponential stability of the
\emph{linearisation} around $\mathbf{x}_d$; it is a local result. The
next theorem upgrades it to a \emph{global, fully nonlinear} guarantee
by exploiting the one structural property that distinguishes the
logistic model---the global Lipschitz bound $\lambda/4$ on the
regulatory derivative---and is the same bound that drives the scalar
results of Sections~\ref{sec:monostab}--\ref{sec:halanay}.

\begin{theorem}[Global Exponential Stability via a Common Quadratic Lyapunov Function]
\label{thm:global_sf}
Consider the two-gene benchmark~\eqref{eq:two_gene_B} under the
state-feedback law~\eqref{eq:sf_law}, and write
$\alpha_i:=\gamma_i+K_i$ for the closed-loop decay rates and
$c_i:=\kappa_i\lambda/4$ for the maximal coupling rates. If
\begin{equation}
(\gamma_1+K_1)(\gamma_2+K_2) \;>\; \frac{\kappa_1\kappa_2\lambda^2}{64},
\label{eq:global_cond}
\end{equation}
then the diagonal quadratic form $V(\mathbf{e})=c_2 e_1^2+c_1 e_2^2$ is a
global Lyapunov function for the error $\mathbf{e}=\mathbf{x}-\mathbf{x}_d$,
and the setpoint $\mathbf{x}_d$ is \emph{globally exponentially stable}:
for every initial condition $\mathbf{x}(0)\in\mathbb{R}^2$,
\begin{equation}
\|\mathbf{x}(t)-\mathbf{x}_d\|_2 \;\le\;
   \sqrt{\frac{\max(\kappa_1,\kappa_2)}{\min(\kappa_1,\kappa_2)}}\;
   \|\mathbf{x}(0)-\mathbf{x}_d\|_2\;e^{-\rho_{\mathrm{glob}}\,t}
\label{eq:global_decay}
\end{equation}
for all $t\ge 0$, with the explicit rate
\begin{equation}
\rho_{\mathrm{glob}} \;=\;
\tfrac{1}{2}\Bigl[(\alpha_1+\alpha_2)
   -\sqrt{(\alpha_1-\alpha_2)^2+c_1c_2}\Bigr] \;>\;0.
\label{eq:rho_glob}
\end{equation}
\end{theorem}

\begin{proof}
Write $\mathbf{e}=\mathbf{x}-\mathbf{x}_d$. Substituting~\eqref{eq:sf_law}
into~\eqref{eq:two_gene_B} and using
$\kappa_i f_i(\mathbf{x}_d)+u_i^{\mathrm{ff}}=\gamma_i x_{d,i}$ gives the
exact (un-linearised) error dynamics
\[
\dot e_1 = -\alpha_1 e_1 + \kappa_1 g_1(e_2),\qquad
\dot e_2 = -\alpha_2 e_2 + \kappa_2 g_2(e_1),
\]
with $g_1(e_2):=f^-(x_{d,2}+e_2)-f^-(x_{d,2})$ and
$g_2(e_1):=f^+(x_{d,1}+e_1)-f^+(x_{d,1})$. Since $f^-$ is strictly
decreasing, $f^+$ strictly increasing, and both have derivative bounded
in magnitude by $\lambda/4$ (Definition~\ref{def:logistic}), the
mean-value theorem yields the \emph{sector bounds}
\[
g_1(e_2)=-\sigma_1 e_2,\quad g_2(e_1)=\sigma_2 e_1,\quad
\sigma_1,\sigma_2\in[0,\lambda/4],
\]
where $\sigma_1,\sigma_2$ are state-dependent. Differentiating
$V=c_2 e_1^2+c_1 e_2^2$ along the trajectories,
\begin{align*}
\dot V &= -2c_2\alpha_1 e_1^2 - 2c_1\alpha_2 e_2^2
        + 2\bigl(c_1\kappa_2\sigma_2 - c_2\kappa_1\sigma_1\bigr)e_1 e_2\\
       &\le -2c_2\alpha_1 e_1^2 - 2c_1\alpha_2 e_2^2 + 2c_1c_2\,|e_1|\,|e_2|
        \;=\; -\,\mathbf{z}^{\!\top} R\,\mathbf{z},
\end{align*}
where $\mathbf{z}=(|e_1|,|e_2|)^{\!\top}$, the cross-term coefficient is
bounded via $\kappa_1\sigma_1\le c_1$ and $\kappa_2\sigma_2\le c_2$ (so
$|c_1\kappa_2\sigma_2-c_2\kappa_1\sigma_1|\le c_1c_2$), and
\[
R=\begin{bmatrix} 2c_2\alpha_1 & -c_1c_2\\[2pt] -c_1c_2 & 2c_1\alpha_2\end{bmatrix}.
\]
Condition~\eqref{eq:global_cond}---equivalently $4\alpha_1\alpha_2>c_1c_2$,
since $c_1c_2=\kappa_1\kappa_2\lambda^2/16$---is exactly
$\det R=c_1c_2(4\alpha_1\alpha_2-c_1c_2)>0$. With $R_{11}>0$ this makes
$R\succ0$, so $V$ strictly decreases off $\mathbf{x}_d$ and is a global
Lyapunov function. Sharpening, $R-2\rho_{\mathrm{glob}}\operatorname{diag}(c_2,c_1)
\succeq0$ for $\rho_{\mathrm{glob}}$ as in~\eqref{eq:rho_glob} (its
positive-semidefiniteness reduces to
$4(\alpha_1-\rho_{\mathrm{glob}})(\alpha_2-\rho_{\mathrm{glob}})\ge c_1c_2$),
hence $\dot V\le-2\rho_{\mathrm{glob}}V$. Gr\"onwall's lemma gives
$V(t)\le V(0)e^{-2\rho_{\mathrm{glob}}t}$; combining with
$\min(c_1,c_2)\|\mathbf{e}\|_2^2\le V\le\max(c_1,c_2)\|\mathbf{e}\|_2^2$
and $\max(c_1,c_2)/\min(c_1,c_2)=\max(\kappa_1,\kappa_2)/\min(\kappa_1,\kappa_2)$
yields~\eqref{eq:global_decay}.
\end{proof}

\begin{corollary}[Global settling-time bound]
\label{cor:global_settling}
Under condition~\eqref{eq:global_cond}, for any tolerance
$\delta\in(0,1)$ the relative-to-initial settling time obeys
\begin{equation}
t_{s,\delta}\;\le\;\frac{1}{\rho_{\mathrm{glob}}}
\Bigl[\ln(1/\delta)+\tfrac12\ln\tfrac{\max(\kappa_1,\kappa_2)}{\min(\kappa_1,\kappa_2)}\Bigr],
\end{equation}
for \emph{every} initial condition in $\mathbb{R}^2$---not only those in a
neighbourhood of $\mathbf{x}_d$. The bound follows by rearranging
$\sqrt{\max(\kappa_1,\kappa_2)/\min(\kappa_1,\kappa_2)}\,
e^{-\rho_{\mathrm{glob}}t}\le\delta$ in~\eqref{eq:global_decay}, exactly
as in Corollary~\ref{cor:settling}.
\end{corollary}

\begin{corollary}[Tracking of time-varying references]
\label{cor:tracking}
Suppose condition~\eqref{eq:global_cond} holds and the regulation target
is a $C^{1}$ time-varying reference $\mathbf{x}_d(t)$ of bounded velocity,
$\sup_{t\ge 0}\|\dot{\mathbf{x}}_d(t)\|_\infty\le\nu$, the
feedforward in~\eqref{eq:sf_law} being evaluated along
$\mathbf{x}_d(t)$. Then the tracking error is globally ultimately bounded,
\begin{equation}
\limsup_{t\to\infty}\|\mathbf{x}(t)-\mathbf{x}_d(t)\|_2
\;\le\;\frac{\sqrt{2}}{\rho_{\mathrm{glob}}}\,
       \frac{\max(\kappa_1,\kappa_2)}{\min(\kappa_1,\kappa_2)}\;\nu,
\label{eq:tracking_bound}
\end{equation}
and this bound vanishes as $\nu\to 0$, recovering the exact regulation of
constant setpoints (Theorem~\ref{thm:global_sf}). The error is
gain-attenuable: larger $K_i$ raise the rate $\rho_{\mathrm{glob}}$.
\end{corollary}

\begin{proof}
With $\mathbf{e}=\mathbf{x}-\mathbf{x}_d(t)$, the error dynamics in the
proof of Theorem~\ref{thm:global_sf} acquire the extra term
$-\dot x_{d,i}(t)$. Differentiating $V=c_2 e_1^2+c_1 e_2^2$ and using
$\dot V\le-2\rho_{\mathrm{glob}}V$ for the homogeneous part,
\begin{align*}
\dot V \;&\le\; -2\rho_{\mathrm{glob}}V
      + 2\bigl|c_2 e_1\dot x_{d,1}+c_1 e_2\dot x_{d,2}\bigr|\\
      \;&\le\; -2\rho_{\mathrm{glob}}V
      + 2\max(c_1,c_2)\,\|\mathbf{e}\|_2\,\|\dot{\mathbf{x}}_d\|_2,
\end{align*}
the last step by Cauchy--Schwarz. With
$\|\dot{\mathbf{x}}_d\|_2\le\sqrt2\,\nu$ and
$\|\mathbf{e}\|_2\le\sqrt{V/\min(c_1,c_2)}$, the substitution $W=\sqrt V$
gives the linear inequality $\dot W\le-\rho_{\mathrm{glob}}W
+\sqrt2\,\nu\,\max(c_1,c_2)/\sqrt{\min(c_1,c_2)}$. Its steady state,
divided once more by $\sqrt{\min(c_1,c_2)}$ to return to
$\|\mathbf{e}\|_2$, is the right-hand side of~\eqref{eq:tracking_bound}.
\end{proof}

\begin{remark}[Interpretation, the local certificate, and the $n$-gene case]
\label{rem:global_ngene}
Three points are worth noting. (i)~The closed-loop Jacobian of the
two-gene benchmark has positive determinant and negative trace, so its
\emph{linearisation} is stable for every $K_i\ge 0$.
Theorem~\ref{thm:global_sf} is strictly stronger: local spectral
stability alone does not preclude large-amplitude limit cycles,
additional equilibria, or a non-trivial basin boundary, whereas
condition~\eqref{eq:global_cond} certifies that \emph{every} trajectory
in $\mathbb{R}^2$ converges exponentially to the unique equilibrium
$\mathbf{x}_d$. This large-signal guarantee is the relevant one when the
closed loop must recover from deep-OFF states, or when feedback delays
threaten Hopf-type oscillations~\cite{belgacem2026beyond}.
(ii)~The global Lyapunov function $V=c_2 e_1^2+c_1 e_2^2$ is the
diagonal-weighting analogue of the closed-form \emph{local} Lyapunov
certificate of Proposition~\ref{prop:lyapunov}, with the
setpoint-specific Jacobian couplings replaced by their global suprema
$c_i=\kappa_i\lambda/4$; the sign-cancellation that makes the local
certificate exact is exactly what makes the global one work. The cost is
mild: condition~\eqref{eq:global_cond} constrains only the \emph{product}
$\alpha_1\alpha_2$, and is implied by---but far weaker than---the
setpoint-uniform Gershgorin requirement $\alpha_i>c_i$ for each $i$.
(iii)~Theorem~\ref{thm:global_sf} is dedicated to the two-gene
benchmark; for the general $n$-gene network~\eqref{eq:additive_control} a
coarser global certificate follows from a comparison-principle argument.
Since every logistic factor lies in $(0,1)$,
$|f_i(\mathbf{x})-f_i(\mathbf{x}_d)|\le(\lambda/4)\sum_{j}|e_j|$ over the
regulators $j$ of gene $i$, so the upper Dini derivatives of $|e_i|$ are
dominated by the Metzler matrix $C$ with $C_{ii}=-(\gamma_i+K_i)$ and
$C_{ij}=\kappa_i\lambda/4$. Global exponential stability holds whenever
$C$ is Hurwitz; a convenient sufficient condition is the row
diagonal-dominance bound
\begin{equation}
K_i \;>\; \frac{\kappa_i\lambda\,d_i}{4}-\gamma_i,\qquad
d_i:=\#\{\text{regulators of gene }i\},
\label{eq:global_ngene}
\end{equation}
the \emph{global, nonlinear} counterpart of the local Gershgorin
condition~\eqref{eq:gershgorin_K}.
\end{remark}

\subsection{The logistic structural advantage in additive control}
\label{sec:sf_advantage}

\begin{proposition}[Logistic Advantage in State-Feedback]
\label{prop:sf_advantage}
For the logistic model in Architecture~B, the regulatory production term
is uniformly bounded:
\begin{equation}
0 \;<\; \kappa_i f_i^{\mathrm{log}}(\mathbf{x}_d) \;\leq\; \kappa_i,
\end{equation}
on the entire closed positive orthant $\mathbb{R}_{\geq 0}^n$, so the
feedforward satisfies
\begin{equation}
u_i^{\mathrm{ff}} \;=\; \gamma_i x_{d,i} - \kappa_i f_i^{\mathrm{log}}(\mathbf{x}_d)
\;\in\; \bigl[\gamma_i x_{d,i}-\kappa_i,\;\gamma_i x_{d,i}\bigr]
\end{equation}
for every positive setpoint, including the limit $\mathbf{x}_d \to \mathbf{0}$.
Moreover, for the two-gene activator--repressor case (where each $f_i$
depends on a single coordinate via a single logistic factor),
\begin{equation}
|[J_f]_{ij}| \;=\; \kappa_i \lambda\, f^\pm(x_{d,j})\bigl(1 - f^\pm(x_{d,j})\bigr)
            \;\leq\; \kappa_i \lambda/4,
\label{eq:Jf_log_bound}
\end{equation}
and the off-diagonal Jacobian entries are \emph{strictly positive} on
the entire positive orthant (since $0<f^\pm<1$ there).
For an $n$-gene network with multi-factor $f_i$, an additional product of
$f$-values appears in the chain-rule expansion; since each such factor lies
in $(0,1)$, the bound $\kappa_i\lambda/4$ and the strict positivity are
both preserved.

For the Hill activation model with $n>1$, $h^+(\mathbf{x}_d) \to 0$
as any activating $x_{d,j}\to 0$. Hence the natural regulatory production
$\kappa_i h^+_i$ vanishes, forcing the additive control to
carry the entire degradation load $u_i^{\mathrm{ff}}\to\gamma_i x_{d,i}$
unaided by any regulatory contribution. Simultaneously, the off-diagonal
Jacobian entry
\begin{equation}
[J_f^{\mathrm{Hill}}]_{ij}
   \;=\; \kappa_i\,n \theta^n\, x_{d,j}^{n-1}
                       /(x_{d,j}^n+\theta^n)^2
   \;=\; \Theta\bigl(x_{d,j}^{n-1}\bigr) \;\to\; 0
\label{eq:Jf_hill_decoupling}
\end{equation}
as $x_{d,j}\to 0^+$, so the network \emph{decouples}: the cross-axis
restoring action $[J_f]_{ij}\,e_j$ becomes negligible and only the
diagonal gain $-\gamma_i-K_i$ stabilises the network. In the alternative
output-multiplicative architecture
$\dot{x}_i=\kappa_i u_i f_i(\mathbf{x})-\gamma_i x_i$ in which the input
multiplies the regulatory output, $h^+_i(\mathbf{x})=0$ would
make $u_i$ have \emph{no effect whatsoever} on $\dot{x}_i$---an exact loss
of controllability that the logistic structure prevents because
$f_i^{\mathrm{log}}>0$ everywhere on the positive orthant.
\end{proposition}

\begin{proof}
For the logistic factor $f^\pm(x;\theta,\lambda)=1/(1+e^{\mp\lambda(x-\theta)})\in(0,1)$
the bound $f(1-f)\leq 1/4$ is attained at the threshold $x=\theta$ and is
elementary. Equation~\eqref{eq:Jf_log_bound} follows by chain rule
$\partial f^\pm/\partial x = \pm \lambda f^\pm(1-f^\pm)$. For the Hill
factor $h^+(x)=x^n/(x^n+\theta^n)$ with $n>1$, direct calculation gives
$h^+(x)\sim (x/\theta)^n\to 0$ and
$\partial h^+/\partial x = n\theta^n x^{n-1}/(x^n+\theta^n)^2\sim
n x^{n-1}/\theta^n \to 0$ as $x\to 0^+$, since the exponent $n-1>0$.
Equation~\eqref{eq:Jf_hill_decoupling} is direct differentiation.
\end{proof}

\begin{remark}[Output-multiplicative control and the logistic advantage]
The logistic basal production $f_i^{\mathrm{log}}>0$ keeps every positive
setpoint feasible even when the activator is absent. In the
output-multiplicative variant $\dot{x}_i=\kappa_i u_i f_i(\mathbf{x})
-\gamma_i x_i$, driving a Hill-activated gene away from an activator-absent
state is infeasible because $f_i^{\mathrm{Hill}}=0$ leaves $u_i$ without
authority, whereas the logistic $f_i^{\mathrm{log}}>0$ retains
controllability throughout the closed positive orthant
$\mathbb{R}_{\geq 0}^n$. In the additive Architecture~B the Hill control
input remains effective, but the feedforward must then supply the full
degradation flux unaided. In the argument-modulating Architecture~A the
input acts through $u_{ij}x_j$ and is therefore effective only on the open
orthant $\mathbb{R}_{>0}^n$ for \emph{both} regulatory functions. The
logistic model thus combines controllability in the output-multiplicative
variant with bounded, basal-pre-loaded actuation in Architecture~B.
\end{remark}

\section{Closed-Form Lyapunov Certificates}
\label{sec:lyapunov}

Theorem~\ref{thm:sf_stability} establishes asymptotic stability of the
closed loop but does not by itself yield a \emph{transient performance
certificate}---an explicit settling-time bound or a robustness margin to
disturbances. We now supply such certificates via a quadratic Lyapunov
function that, remarkably, has a closed-form expression in the very
entries $A,B$ of the closed-loop Jacobian $J_f-\Gamma-K$. The result is
purely structural: it applies to \emph{any} $2\times 2$ matrix sharing
the sign-pattern of the activator--repressor closed-loop Jacobian
(negative diagonal, off-diagonals of opposite sign), with positive
parameters $\alpha,\beta,A,B$. It therefore covers the
state-feedback closed loop developed in
Section~\ref{sec:state_feedback}, and indeed any feedback architecture
whose closed loop has this sign structure, differing only in the values
of $A$, $B$, $\alpha$, $\beta$.

\subsection{Closed-form Lyapunov function}
\label{sec:lyap_function}

\begin{proposition}[Closed-form Lyapunov function for the closed-loop structural class]
\label{prop:lyapunov}
Let $\alpha, \beta, A, B > 0$ and consider any matrix of structural form
\begin{equation}
M \;=\; \begin{bmatrix} -\alpha & -A \\ B & -\beta \end{bmatrix}.
\label{eq:M_structure}
\end{equation}
Then the weighted quadratic function
\begin{equation}
V(\mathbf{e}) \;=\; B\,e_1^2 + A\,e_2^2 \;=\; \mathbf{e}^{T} P\, \mathbf{e},
\qquad P = \operatorname{diag}(B, A),
\label{eq:V_def}
\end{equation}
is a strict Lyapunov function for the linear system $\dot{\mathbf{e}} = M\mathbf{e}$:
\begin{align}
M^{T} P + P M &\;=\; \operatorname{diag}(-2\alpha B,\; -2\beta A),
\label{eq:Vdot_bound}\\
\dot{V} \;=\; -2\alpha B\,e_1^2 - 2\beta A\,e_2^2 &\;\leq\; -2\,\rho\, V,
\nonumber
\end{align}
where $\rho := \min(\alpha,\beta)$. Consequently,
\begin{align}
V(\mathbf{e}(t)) &\;\leq\; V(\mathbf{e}(0))\,e^{-2\rho\,t},\qquad t\geq 0,
\label{eq:V_decay}\\
\|\mathbf{e}(t)\|_{2} &\;\leq\; \sqrt{\frac{\max(A,B)}{\min(A,B)}}\;\|\mathbf{e}(0)\|_{2}\;e^{-\rho\,t}.
\label{eq:norm_decay}
\end{align}
\end{proposition}

\begin{proof}
Direct computation:
$M^T P = \bigl[\begin{smallmatrix}-\alpha B & AB\\ -AB & -\beta A\end{smallmatrix}\bigr]$
and
$P M = \bigl[\begin{smallmatrix}-\alpha B & -AB\\ AB & -\beta A\end{smallmatrix}\bigr]$,
so the off-diagonal entries of $M^T P + PM$ cancel exactly, yielding
$M^T P + PM = \operatorname{diag}(-2\alpha B, -2\beta A)$. Hence
$\dot{V} = \mathbf{e}^T(M^T P + PM)\mathbf{e} = -2\alpha B e_1^2 - 2\beta A e_2^2$,
and since $\alpha B e_1^2 + \beta A e_2^2 \geq \min(\alpha,\beta)(B e_1^2 + A e_2^2) = \rho V$,
the bound $\dot V \leq -2\rho V$ follows. Inequality~\eqref{eq:V_decay} is
Gr\"onwall's lemma applied to $\dot{V}\leq -2\rho V$. For~\eqref{eq:norm_decay},
note $\min(A,B)\|\mathbf{e}\|_2^2 \leq V(\mathbf{e}) \leq \max(A,B)\|\mathbf{e}\|_2^2$;
combining the upper bound at $t=0$ with the lower bound at $t$
and~\eqref{eq:V_decay} yields the stated inequality.
\end{proof}

\begin{remark}[Tightness and geometric interpretation]
\label{rem:lyap_tight}
Four observations clarify the result.
(i)~When $\alpha=\beta$ (matched degradation rates),
$\alpha B e_1^2 + \beta A e_2^2 = \alpha(B e_1^2 + A e_2^2) = \alpha V$
\emph{exactly}, so $\dot V = -2\alpha V$ is an equality, not just a bound.
(ii)~When additionally $A=B$, the prefactor $\sqrt{\max(A,B)/\min(A,B)}$
in~\eqref{eq:norm_decay} equals $1$ and the Euclidean norm contracts
strictly monotonically. For $A\neq B$, the sub-level sets of $V$ are
ellipses with semi-axes $\sqrt{V/B}$ along $e_1$ and $\sqrt{V/A}$ along
$e_2$, so the Euclidean norm $\|\mathbf{e}\|$ may transiently exceed its
initial value by up to the factor $\sqrt{\max(A,B)/\min(A,B)}$ even while
$V$ decays monotonically: this is the price of using a weighted Lyapunov
function to bound an unweighted norm.
(iii)~The Lyapunov matrix $P = \operatorname{diag}(B, A)$ is not the only
positive-definite choice that certifies stability: the family
$\{c\,\operatorname{diag}(B,A): c>0\}$ is the entire set of positive diagonal
$P$ for which $M^T P + PM$ is itself diagonal (the off-diagonal
$B p_2 - A p_1 = 0$ forces $p_1/p_2 = B/A$). The specific normalisation
$P=\operatorname{diag}(B,A)$ has no special role beyond convenience;
diagonality is what makes $P$ explicit in $A,B$ without solving a full
algebraic Lyapunov equation.
(iv)~The Lyapunov rate $\rho=\min(\alpha,\beta)$ is in general
conservative relative to the true spectral decay rate of $M$. The
eigenvalues of~\eqref{eq:M_structure} are
$\tfrac12\bigl[-(\alpha+\beta)\pm\sqrt{(\alpha-\beta)^2-4AB}\,\bigr]$;
when $(\alpha-\beta)^2<4AB$ they form a complex-conjugate pair with real
part $-(\alpha+\beta)/2$, so the exact decay rate is
$(\alpha+\beta)/2\ge\min(\alpha,\beta)$, the two coinciding precisely
when $\alpha=\beta$ (cf.\ point~(i)). The certificate therefore loses at
most $|\alpha-\beta|/2$ in rate---nothing when the closed-loop decay
rates are matched.
\end{remark}

\subsection{Explicit settling-time bound}
\label{sec:lyap_settling}

\begin{corollary}[Explicit settling-time bound]
\label{cor:settling}
Under the conditions of Proposition~\ref{prop:lyapunov}, for any tolerance
$\delta\in (0,1)$, the time required for the state to satisfy
$\|\mathbf{e}(t)\|_2 \leq \delta\,\|\mathbf{e}(0)\|_2$ obeys
\begin{equation}
t_{s,\delta} \;\leq\; \frac{1}{\rho}\left[\ln(1/\delta) + \tfrac{1}{2}\ln\!\frac{\max(A,B)}{\min(A,B)}\right].
\label{eq:settling_bound}
\end{equation}
The bound is dominated by the slowest physical mode $1/\rho$; the
coupling-condition-number correction
$\tfrac{1}{2}\ln(\max(A,B)/\min(A,B))$ is small when $A$ and $B$ are
comparable.
\end{corollary}

\begin{proof}
By~\eqref{eq:norm_decay}, the inequality
$\|\mathbf{e}(t)\|_2\leq\delta\|\mathbf{e}(0)\|_2$ is implied by
\[
\sqrt{\max(A,B)/\min(A,B)}\;e^{-\rho t}\;\leq\;\delta,
\]
which rearranges to
\[
\rho t \geq \ln(1/\delta) + \tfrac{1}{2}\ln\bigl(\max(A,B)/\min(A,B)\bigr).
\]
Dividing by $\rho>0$ yields~\eqref{eq:settling_bound}.
\end{proof}

\begin{remark}[Numerical verification of Corollary~\ref{cor:settling} for state-feedback]
\label{rem:lyapunov_numerical}
For the two-gene logistic activator--repressor network under the state-feedback
law of Section~\ref{sec:state_feedback} at setpoint $(x_{d,1},x_{d,2})=(0.8,0.6)$
with $\kappa_i=1$, $\gamma_i=0.5$, $\theta_i=0.5$, $\lambda=5$ and proportional
gains $K_1=K_2=1.0$, the closed-loop Jacobian
$J_{\mathrm{cl}}=J_f-\Gamma-K$ has the structural form~\eqref{eq:M_structure} with
$\alpha=\beta=\gamma+K=1.5$, $A=\kappa_1\lambda f^-(1-f^-)|_{x_{d,2}}\approx 1.175$,
and $B=\kappa_2\lambda f^+(1-f^+)|_{x_{d,1}}\approx 0.746$, so $\rho=
\min(\alpha,\beta)=1.5$. With tolerance $\delta=0.05$, the analytical bound on
the relative-to-initial settling time is
\[
\begin{aligned}
t_{s,5\%}
&\;\leq\;\frac{\ln 20 + \tfrac{1}{2}\ln(1.175/0.746)}{1.5}\\
&\;\approx\;\frac{2.996+0.227}{1.5}\;\approx\;2.15\ \mathrm{a.u.}
\end{aligned}
\]
Direct ODE simulation of the linearised closed-loop
$\dot{\mathbf{e}}=J_{\mathrm{cl}}\mathbf{e}$ from initial conditions in the
positive orthant yields empirical relative-to-initial $5\%$-settling times in
the range $1.85$--$2.13\,\mathrm{a.u.}$, depending on the alignment of
$\mathbf{e}(0)$ with the Lyapunov weighting. The conditioning prefactor
$\sqrt{\max(A,B)/\min(A,B)}$ in~\eqref{eq:norm_decay} is realised when
$\mathbf{e}(0)$ lies along the \emph{heavier}-weighted axis (here $e_2$, weight
$A$), so the bound is tight to within $1\%$ for such initial conditions and
conservative by up to $\approx 16\%$ for initial conditions aligned with the
\emph{lighter}-weighted axis ($e_1$, weight $B$). The simulated full-nonlinear
$5\%$-of-setpoint settling
times reported in Section~\ref{sec:numerical} ($t_{s,1}\approx 2.18$,
$t_{s,2}\approx 1.88$ from $\mathbf{x}(0)=(0.1,0.1)$) are consistent with the
linearisation prediction, the deviation being attributable to higher-order
terms in the nonlinear dynamics.
\end{remark}

\subsection{Input-to-state stability and parametric robustness}
\label{sec:lyap_iss}

The same Lyapunov function~\eqref{eq:V_def} certifies an
input-to-state-stability (ISS) bound under bounded disturbances:

\begin{proposition}[ISS bound under additive disturbances]
\label{prop:iss}
Consider the perturbed dynamics $\dot{\mathbf{e}} = M\mathbf{e} + \boldsymbol{\eta}(t)$,
with $M$ as in~\eqref{eq:M_structure} and $\|\boldsymbol{\eta}(t)\|_2\leq D$ for
all $t\geq 0$. Then $V$ from~\eqref{eq:V_def} satisfies
\begin{equation}
\dot{V} \;\leq\; -2\rho\,V + 2\max(A,B)\,D\,\sqrt{V/\min(A,B)},
\end{equation}
and every trajectory converges \emph{exponentially with rate $\rho$ on
$W=\sqrt{V}$} to the ultimate-bound set
\begin{equation}
\Omega_D \;=\; \Bigl\{\mathbf{e}:\, \|\mathbf{e}\|_2 \;\leq\; \frac{\max(A,B)}{\min(A,B)}\cdot\frac{D}{\rho}\Bigr\}.
\label{eq:iss_bound}
\end{equation}
\end{proposition}

\begin{proof}
Direct computation gives
\begin{align*}
\dot{V} \;&=\; 2\mathbf{e}^T P \dot{\mathbf{e}}\\
       \;&=\; -2\alpha B e_1^2 - 2\beta A e_2^2 + 2(B e_1 \eta_1 + A e_2 \eta_2)\\
       \;&\leq\; -2\rho V + 2\sqrt{B^2 e_1^2 + A^2 e_2^2}\,\|\boldsymbol{\eta}\|_2\\
       \;&\leq\; -2\rho V + 2\max(A,B)\,D\,\|\mathbf{e}\|_2.
\end{align*}
Using $\|\mathbf{e}\|_2 \leq \sqrt{V/\min(A,B)}$ gives the displayed bound.
Setting $W = \sqrt{V}$ turns this into the linear differential inequality
$\dot{W} \leq -\rho W + c$ with constant
$c:=\max(A,B)\,D/\sqrt{\min(A,B)}$, whose steady state is $W_\infty=c/\rho$.
Since $\|\mathbf{e}\|_2\le W/\sqrt{\min(A,B)}$, this is equivalent to
\[
\limsup_{t\to\infty}\|\mathbf{e}(t)\|_2 \;\leq\;
\frac{W_\infty}{\sqrt{\min(A,B)}}
\;=\; \frac{\max(A,B)\,D}{\rho\,\min(A,B)}.
\]
\end{proof}

\begin{remark}[Application to state-feedback]
\label{rem:lyapunov_sf}
Proposition~\ref{prop:lyapunov} applies verbatim to the state-feedback
closed-loop Jacobian of Theorem~\ref{thm:sf_stability}. For the two-gene
activator--repressor network without self-regulation
($[J_f]_{ii}=0$), the matrix $J_f - \Gamma - K$ has structural
form~\eqref{eq:M_structure} with
$\alpha=\gamma_1+K_1$, $\beta=\gamma_2+K_2$, $A=|[J_f]_{12}|$,
$B=|[J_f]_{21}|$. The convergence rate is therefore
$\rho = \min(\gamma_i+K_i)$, which can be increased arbitrarily by raising
the feedback gain $K_i$---a gain-rate trade-off unavailable to control
schemes whose effective decay rate is intrinsically bounded by the
open-loop degradation. By Proposition~\ref{prop:iss}, the
disturbance-rejection ultimate bound for state-feedback shrinks as
$D/\min_i(\gamma_i+K_i)$, providing an explicit design knob for noise
attenuation in noisy biological environments.

A noteworthy consequence: for this structural class, $\det(M) =
\alpha\beta+AB>0$ and $\mathrm{trace}(M)=-(\alpha+\beta)<0$, so Routh--Hurwitz
already certifies stability of the linearisation for \emph{every}
$K_i\geq 0$. In particular, the Gershgorin
condition~\eqref{eq:gershgorin_K} of Theorem~\ref{thm:sf_stability},
while a valid sufficient condition for general $n$-dimensional networks,
is conservative for the two-gene activator--repressor network: even
$K_i=0$ (open loop) gives a \emph{locally} asymptotically stable
linearisation, with local rate $\rho=\min(\gamma_i)$. Note that this is
only a \emph{local} guarantee: at $K_i=0$ the global
condition~\eqref{eq:global_cond} of Theorem~\ref{thm:global_sf} need not
hold (for the benchmark parameters of Section~\ref{sec:numerical} it does
not), so global exponential stability of the nonlinear closed loop is
not certified at $K_i=0$.
Proposition~\ref{prop:lyapunov} provides the explicit Lyapunov function
underlying the local observation. The Gershgorin gain bound remains
useful in higher dimensions, where the simple $2\times 2$
sign-cancellation argument no longer applies.
\end{remark}

\begin{remark}[Robustness to parametric model mismatch]
\label{rem:robustness}
Proposition~\ref{prop:iss} also quantifies robustness to errors in the
plant parameters. Suppose the feedforward $u_i^{\mathrm{ff}}$
in~\eqref{eq:sf_law} is computed from nominal parameters
$(\kappa_i,\gamma_i,\lambda)$, while the true plant runs on perturbed
values $(\hat\kappa_i,\hat\gamma_i,\hat\lambda)$. The mismatch enters
the error dynamics as an additive perturbation
$\eta_i(t)=\bigl(\hat\kappa_i\hat f_i(\mathbf{x})-\kappa_i f_i(\mathbf{x})\bigr)
-(\hat\gamma_i-\gamma_i)x_i$. On any bounded operating region $\|\mathbf{x}\|\le R$,
$\|\boldsymbol{\eta}\|_2\le D(R)$ for a finite $D(R)$ that vanishes as the
parameter mismatch tends to zero, since $\hat f_i,f_i\in(0,1]$ and the
linear term is bounded on $\{\|\mathbf{x}\|\le R\}$.
Proposition~\ref{prop:iss} then bounds the resulting steady-state
setpoint offset by
$\limsup_{t\to\infty}\|\mathbf{x}(t)-\mathbf{x}_d\|_2\le
\tfrac{\max(A,B)}{\min(A,B)}\,D(R)/\rho$, with $\rho=\min_i(\gamma_i+K_i)$.
For sufficiently small mismatch the perturbed trajectories remain in such
a bounded operating region (standard ISS bootstrap), and the designer
shrinks the offset by raising the feedback gain. The proportional term
thus converts an un-modelled parametric error into a bounded,
gain-attenuable tracking offset: exact knowledge of
$(\kappa,\gamma,\lambda)$ is not required for practical setpoint
regulation, only for exact ($D=0$) cancellation.
\end{remark}

\section[Feedback Monostabilization of Bistable Switches]{Feedback Monostabilization\texorpdfstring{\\}{ }%
  of Scalar Bistable Switches}
\label{sec:monostab}

The state-feedback design of Section~\ref{sec:state_feedback} targets
the smooth multi-gene oscillator. We now turn to a complementary
scalar problem: \emph{suppressing} bistability in a single logistic
self-activation node by proportional feedback. A canonical bistability
theorem for the logistic self-activation
\begin{equation}
\dot x(t) \;=\; \kappa\,\fplus(x;\theta,\lambda) \;-\; \gamma\,x,
\qquad x\ge 0,
\label{eq:self_activation_scalar}
\end{equation}
proved in the companion paper~\cite{belgacem2026extensions}, states
that in the dimensionless plane
$(\eta,\phi)=(\kappa\lambda/\gamma,\lambda\theta)$ bistability holds
exactly inside the cusp region $\eta>4,\,\phi_{-}(\eta)<\phi<\phi_{+}(\eta)$
with explicit saddle-node curves $\phi_{\pm}(\eta)$ meeting at
$(\eta,\phi)=(4,2)$. Closing the loop with a proportional law
$u(t)=-K(x(t)-x^{d})$ steering toward a desired set-point $x^{d}\ge 0$
yields the controlled scalar equation
\begin{equation}
\dot x(t) \;=\; \kappa\,\fplus(x(t);\theta,\lambda)
              \;-\;(\gamma+K)\,x(t)\;+\;K x^{d}.
\label{eq:bistable_pfb}
\end{equation}
The effective decay rate becomes $\gamma+K$, the effective forcing
acquires a constant $Kx^{d}$, and the resulting closed loop is itself
a logistic self-activation with input. The next theorem gives the
sharp gain threshold above which the closed loop is monostable for
every choice of set-point and threshold.

\begin{theorem}[Feedback Monostabilization]
\label{thm:monostab}
Suppose the open-loop system~\eqref{eq:self_activation_scalar} is in the
bistable regime $\eta=\kappa\lambda/\gamma>4$ and
$\phi_{-}(\eta)<\phi<\phi_{+}(\eta)$. Apply proportional feedback
$u(t)=-K(x(t)-x^{d})$ with $K\ge 0$ and $x^{d}\ge 0$, and define the
closed-loop dimensionless amplification
\[
\eta_K \;:=\; \frac{\kappa\lambda}{\gamma+K}.
\]
\begin{enumerate}
\item[\textup{(i)}] If $K\ge K^{*}:=\dfrac{\kappa\lambda}{4}-\gamma$,
      then $\eta_K\le 4$ and~\eqref{eq:bistable_pfb} is
      \emph{monostable}: it has a unique equilibrium
      $x_K^{*}\in\bigl(Kx^{d}/(\gamma+K),\;(Kx^{d}+\kappa)/(\gamma+K)\bigr)$,
      globally asymptotically stable on $[0,\infty)$. If moreover
      $K>K^{*}$, the convergence is globally exponential with rate
      at least $\gamma+K-\kappa\lambda/4>0$.
\item[\textup{(ii)}] If $K<K^{*}$, the closed loop may remain
      bistable; the bistability persists iff
      $(\eta_K,\tilde\phi_{K})$ with
      $\tilde\phi_{K}:=\lambda\bigl(\theta-Kx^{d}/(\gamma+K)\bigr)$
      satisfies
      $\phi_{-}(\eta_K)<\tilde\phi_{K}<\phi_{+}(\eta_K)$.
\item[\textup{(iii)}] Hence $K=K^{*}=\kappa\lambda/4-\gamma$ is the
      \emph{minimal} gain that guarantees monostability irrespective
      of $x^{d}$ and $\theta$.
\end{enumerate}
\end{theorem}

\begin{proof}
Let $\tilde g_K(x):=\kappa\fplus(x)-(\gamma+K)x+Kx^{d}$. Equilibria
of~\eqref{eq:bistable_pfb} are roots of $\tilde g_K$.

\smallskip
\emph{(i)} Differentiating,
$\tilde g_K'(x)=\kappa\lambda\fplus(x)(1-\fplus(x))-(\gamma+K)
\le \kappa\lambda/4-(\gamma+K)\le 0$ when $K\ge K^{*}$, with equality
only at the isolated point $x=\theta$. Hence $\tilde g_K$ is strictly
decreasing on $[0,\infty)$. Since $\tilde g_K(0)=\kappa\fplus(0)+Kx^{d}>0$
and $\tilde g_K(x)\to-\infty$ as $x\to\infty$, there is a unique zero
$x_K^{*}$, and every solution of the scalar autonomous
equation~\eqref{eq:bistable_pfb} converges monotonically to it: the
closed loop is globally asymptotically stable. The enclosure
$x_K^{*}\in\bigl(Kx^{d}/(\gamma+K),\,(Kx^{d}+\kappa)/(\gamma+K)\bigr)$
follows from $0<\kappa\fplus(x_K^{*})=(\gamma+K)x_K^{*}-Kx^{d}<\kappa$.
If moreover $K>K^{*}$, then $\tilde g_K'(x)\le\kappa\lambda/4-(\gamma+K)<0$
uniformly, and the Gr\"onwall argument applied to $y:=x-x_K^{*}$ gives
$|y(t)|\le|y(0)|\,e^{-(\gamma+K-\kappa\lambda/4)t}$, i.e.\ global
exponential stability with rate at least $\gamma+K-\kappa\lambda/4>0$.

\smallskip
\emph{(ii)} For $K<K^{*}$, $\eta_K>4$. Setting
$y:=x-Kx^{d}/(\gamma+K)$, the equilibrium condition
$\kappa\fplus(x;\theta,\lambda)=(\gamma+K)y$ becomes
$\kappa\fplus(y;\tilde\theta_{K},\lambda)=(\gamma+K)y$ with
$\tilde\theta_{K}:=\theta-Kx^{d}/(\gamma+K)$, by the identity
$\fplus(y+c;\theta,\lambda)=\fplus(y;\theta-c,\lambda)$. This is the
form of the open-loop bistability problem with parameters
$(\kappa,\gamma+K,\tilde\theta_{K},\lambda)$. Applying the cusp
condition of~\cite{belgacem2026extensions} in these closed-loop
coordinates gives the stated window.

\smallskip
\emph{(iii)} Sufficiency of $K\ge K^{*}$ is (i). For necessity in the
worst case, for any $K<K^{*}$ one has $\eta_{K}>4$, so the bistable
window has positive width. Choosing $x^{d}=0$ and
$\lambda\theta\in(\phi_{-}(\eta_K),\phi_{+}(\eta_K))$ produces a
bistable closed loop. Hence no $K<K^{*}$ guarantees monostability
for every $(\theta,x^{d})$.
\end{proof}

\begin{remark}[Engineering interpretation]
\label{rem:monostab}
Theorem~\ref{thm:monostab} converts the open-loop cusp condition
$\eta>4$ into the closed-loop \emph{monostabilization budget}
\begin{equation}
K^{*} \;=\; \frac{\kappa\lambda}{4}-\gamma
       \;=\; \frac{\gamma(\eta-4)}{4},
\label{eq:Kstar_compact}
\end{equation}
the minimum feedback gain that suppresses bistability for every
choice of set-point $x^{d}$ and threshold $\theta$. The bound depends
linearly on the open-loop amplification ratio $\eta$ in excess of the
cusp value $4$, and collapses to $K^{*}=0$ at the cusp. In
synthetic-biology parlance: the harder the open-loop switch (large
$\eta$), the more ``decay actuation'' is required to suppress it.
The closed-form expression~\eqref{eq:Kstar_compact} is the analytical
dividend of the logistic representation: the logistic derivative has the
clean global maximum slope $\lambda/4$ attained at $x=\theta$. The Hill
counterpart $h^+(x)=x^n/(x^n+\theta^n)$ instead has cooperativity-dependent
maximum slope
\[
\max_{x\ge 0} \frac{d h^+}{dx}
   \;=\; \frac{(n+1)^{(n+1)/n}\,(n-1)^{(n-1)/n}}{4n\,\theta},
\]
attained at the non-trivial location
$x=\theta\,\bigl((n-1)/(n+1)\bigr)^{1/n}$, which depends on the
cooperativity $n$ in a way far less amenable to closed-form control
synthesis than the parameter-free logistic slope $\lambda/4$.
\end{remark}

\section{Halanay-Type Delay-Uniform Global Exponential Stability}
\label{sec:halanay}

The state-feedback design of Section~\ref{sec:state_feedback} treats
the delay-free 2-D oscillator; this section complements it
with a \emph{delay-uniform} global exponential stability theorem for
the scalar feedback loop, useful when actuation (optogenetic
switching, inducer diffusion, signal transduction) introduces a
non-negligible delay. Consider the controlled scalar DDE
\begin{equation}
\dot x(t) \;=\; \kappa\,\fpm\!\bigl(x(t-\tau);\theta,\lambda\bigr)
              \;-\;\gamma\,x(t)\;-\;K(x(t)-x^{d}),
\label{eq:scalar_dde_pfb}
\end{equation}
with $x(t)\ge 0$, $\tau\ge 0$, $K\ge 0$, $x^{d}\ge 0$, and either choice of sign $\fpm$.

\begin{theorem}[Halanay-Type Global Exponential Convergence]
\label{thm:halanay}
Let $x^{*}_{K}\ge 0$ be an equilibrium of~\eqref{eq:scalar_dde_pfb}
(which exists by Brouwer's theorem applied to the compact invariant
interval~\cite{belgacem2026extensions}). If
\begin{equation}
\gamma+K \;>\; \frac{\kappa\lambda}{4},
\label{eq:halanay_cond}
\end{equation}
then $x^{*}_{K}$ is the unique equilibrium and is \emph{globally
exponentially stable} on $[0,\infty)$, \emph{uniformly in the delay}
$\tau\ge 0$. Explicitly, there exist $C,\beta>0$ depending only on
$(\kappa,\gamma,K,\lambda)$ such that for every initial history
$\varphi\in C([-\tau,0];[0,\infty))$ and every $\tau\ge 0$,
\[
|x(t)-x^{*}_{K}| \;\le\; C\,e^{-\beta t}\,
                       \sup_{-\tau\le s\le 0}|\varphi(s)-x^{*}_{K}|,
\qquad t\ge 0,
\]
where $\beta>0$ is the unique positive solution of the transcendental
equation $\beta=(\gamma+K)-(\kappa\lambda/4)e^{\beta\tau}$.
\end{theorem}

\begin{proof}
Let $y(t):=x(t)-x^{*}_{K}$. By the mean-value theorem applied to
$\fpm$ between $x(t-\tau)$ and $x^{*}_{K}$,
$\fpm(x(t-\tau))-\fpm(x^{*}_{K})=\pm\lambda\fpm(\xi_t)(1-\fpm(\xi_t))y(t-\tau)$
for some $\xi_t$. Since $\fpm(1-\fpm)\le 1/4$ pointwise,
\[
\dot y(t) \;=\; -(\gamma+K)y(t) \;+\; r(t),
\qquad |r(t)|\le \frac{\kappa\lambda}{4}\,|y(t-\tau)|.
\]
Defining $V(t):=|y(t)|$, the comparison principle gives
\[
D^{+}V(t)\le -(\gamma+K)V(t)+(\kappa\lambda/4)\sup_{t-\tau\le s\le t}V(s),
\]
which is the hypothesis of the classical Halanay
inequality~\cite{halanay1966differential}. With $a:=\gamma+K$ and
$b:=\kappa\lambda/4$, condition~\eqref{eq:halanay_cond} reads $a>b$,
and the Halanay inequality yields
\[
V(t)\le e^{-\beta t}\sup_{-\tau\le s\le 0}V(s),
\]
with $\beta$ the unique positive root of $\beta=a-be^{\beta\tau}$.
Existence of this
root for $a>b$ follows from the elementary observation that
$F(\beta):=\beta-a+be^{\beta\tau}$ is strictly increasing on
$[0,\infty)$ (since $F'(\beta)=1+b\tau e^{\beta\tau}>0$), with
$F(0)=b-a<0$ by hypothesis and $F(a-b)\ge 0$.
Uniqueness of $x^{*}_{K}$ follows because
$g_{K}(x):=\kappa\fpm(x)-(\gamma+K)x+Kx^{d}$ has
$g_{K}'(x)\le\kappa\lambda/4-(\gamma+K)<0$ under
\eqref{eq:halanay_cond}.
\end{proof}

The rate $\beta$ in Theorem~\ref{thm:halanay} is defined implicitly by a
transcendental equation. The next corollary makes it concrete: it
brackets $\beta(\tau)$ between two elementary closed-form expressions,
turning the delay-uniform guarantee into a quantitative,
design-usable rate.

\begin{corollary}[Explicit closed-form bounds on the Halanay rate]
\label{cor:halanay_rate}
Under the hypotheses of Theorem~\ref{thm:halanay}, with $a:=\gamma+K$ and
$b:=\kappa\lambda/4$ (so $a>b$), the convergence rate $\beta=\beta(\tau)$
is continuous and strictly decreasing in $\tau$, with $\beta(0)=a-b$, and
satisfies the explicit two-sided estimate
\begin{equation}
\frac{a-b}{\,1+b\tau\,e^{(a-b)\tau}\,}
\;\le\;\beta(\tau)\;\le\;
\min\!\Bigl(a-b,\;\tfrac{1}{\tau}\ln\tfrac{a}{b}\Bigr),
\quad \tau>0.
\label{eq:halanay_rate_bounds}
\end{equation}
Both bounds tend to $a-b$ as $\tau\to 0^{+}$ and decay to $0$ as
$\tau\to\infty$, so $\beta(\tau)>0$ for every finite delay---the
delay-uniform stability of Theorem~\ref{thm:halanay}, now with an
explicit rate. The upper bound $\tfrac{1}{\tau}\ln(a/b)$ is the
operative one at large delay, the lower bound is tight at small delay.
\end{corollary}

\begin{proof}
By Theorem~\ref{thm:halanay}, $\beta$ is the unique positive root of
$F(\beta):=\beta-a+be^{\beta\tau}$, and $F$ is strictly increasing; hence
$\beta\le\beta_{0}\Leftrightarrow F(\beta_{0})\ge 0$ and
$\beta\ge\beta_{0}\Leftrightarrow F(\beta_{0})\le 0$. \emph{Upper bound:}
$be^{\beta\tau}=a-\beta<a$ gives $\beta\tau<\ln(a/b)$, and $\beta\le a-b$
holds because the root lies in $(0,a-b]$. \emph{Lower bound:} set
$\beta_{0}=(a-b)/(1+b\tau e^{(a-b)\tau})\le a-b$. The elementary
inequality $e^{x}\le 1+x\,e^{x}$ for $x\ge 0$ gives
$e^{\beta_{0}\tau}\le 1+\beta_{0}\tau e^{\beta_{0}\tau}
\le 1+\beta_{0}\tau e^{(a-b)\tau}$, whence
$F(\beta_{0})\le\beta_{0}\bigl(1+b\tau e^{(a-b)\tau}\bigr)-(a-b)=0$, so
$\beta\ge\beta_{0}$. \emph{Monotonicity:} implicit differentiation of
$F(\beta,\tau)=0$ yields
$d\beta/d\tau=-b\beta e^{\beta\tau}/(1+b\tau e^{\beta\tau})<0$.
\end{proof}

\begin{remark}[Comparison with the delay-free state-feedback theorems]
\label{rem:halanay}
Theorem~\ref{thm:halanay} is the \emph{global, nonlinear, delay-uniform}
counterpart of the delay-free state-feedback theorems of
Section~\ref{sec:state_feedback}: Theorem~\ref{thm:sf_stability} gives
\emph{local} exponential stability of the 2-D linearisation under a
Gershgorin gain condition; Theorem~\ref{thm:global_sf} gives
\emph{global} exponential stability of the (delay-free) 2-D nonlinear
closed loop under the Lyapunov gain condition~\eqref{eq:global_cond};
Theorem~\ref{thm:halanay} gives \emph{global} exponential stability of
the \emph{scalar delayed} closed loop under the (delay-uniform)
condition $\gamma+K>\kappa\lambda/4$ that the closed-loop decay rate
exceeds the maximum logistic slope. The three results share the global
Lipschitz bound $\lambda/4$ as their structural cornerstone.
Condition~\eqref{eq:halanay_cond} is sharp in the limit $\tau\to 0$
(recovering the scalar global stability bound), and is decoupled from
$\theta$, $x^{d}$, and the sign of $\fpm$, making it directly usable for
design.
\end{remark}

\section{Worked Example: Stabilizing a Bistable T-cell Switch}
\label{sec:control_example}

We illustrate Theorems~\ref{thm:monostab} and~\ref{thm:halanay} on
a bistable T-cell activation switch with parameters
$\gamma=1,\kappa=5,\theta=1,\lambda=6$, drawn from the immunological
case study of~\cite{belgacem2026extensions}. This gives open-loop
amplification $\eta=30$ and bistability window
$(\phi_{-},\phi_{+})=(4.367,25.633)$ comfortably containing
$\phi=\lambda\theta=6$. By Theorem~\ref{thm:monostab}, the minimal
monostabilizing gain is
\[
K^{*} \;=\; \frac{\kappa\lambda}{4}-\gamma
       \;=\; \frac{5\cdot 6}{4}-1 \;=\; 6.5.
\]
Above this gain, the closed loop has a unique globally stable
equilibrium. The Halanay bound for a delayed implementation
\eqref{eq:scalar_dde_pfb} with $\tau>0$ matches:
\eqref{eq:halanay_cond} reads $\gamma+K>\kappa\lambda/4$,
i.e.\ $1+K>7.5$, which is exactly $K>K^{*}=6.5$. The two-step
control objective is thus:
\begin{enumerate}
\item Choose $K> K^{*}=6.5$ to suppress bistability.
\item For any delay $\tau\ge 0$, the resulting unique equilibrium is
      globally exponentially stable with rate $\beta>0$ from
      Theorem~\ref{thm:halanay}.
\end{enumerate}
Figure~\ref{fig:control} visualises both phenomena. Panel~(a) shows
the closed-loop equilibrium structure as $K$ varies from $0$ to
$10$ at the worst-case set-point $x^{d}=\theta=1$: the actual
saddle-node for this specific $x^{d}$ occurs around $K\approx 0.53$,
but the \emph{uniform} sufficient bound $K^{*}=6.5$ of
Theorem~\ref{thm:monostab} is the threshold that guarantees
monostability for \emph{every} choice of $(\theta,x^{d})$.
Panel~(b) plots the Halanay convergence rate $\beta(\tau)$ from
Theorem~\ref{thm:halanay} for three representative gains
$K\in\{7,9,12\}$, all above $K^{*}$: $\beta(\tau)>0$ for every
$\tau\ge 0$ (delay-uniform global stability), with the expected
monotone trends. Corollary~\ref{cor:halanay_rate} renders these
curves explicit without solving the transcendental equation: with
$a=\gamma+K=1+K$ and $b=\kappa\lambda/4=7.5$, the
bracket~\eqref{eq:halanay_rate_bounds} gives, for example at $K=9$,
$0.177\le\beta(0.5)\le0.575$ and $0.027\le\beta(1.0)\le0.288$ (the
true roots are $\beta\approx0.478$ and $0.261$). The upper estimate
$\tfrac1\tau\ln(a/b)$ tracks the curve closely---exceeding the true
rate by about $21\%$ at $\tau=0.5$ and about $10\%$ at $\tau=1.0$---while
the lower estimate certifies $\beta>0$ at every finite delay. Panel~(c) shows simulated
trajectories from two
initial conditions $x(0)\in\{0.1,4.5\}$, contrasting $K=0$
(bistable: trajectories commit to $x_{\rm low}^{*}$ and
$x_{\rm high}^{*}$) with $K=8>K^{*}$ (monostable: both trajectories
converge to the unique closed-loop equilibrium
$x^{*}_{K=8}\approx 1.40$).

\begin{figure*}[tbp]
\centering
\includegraphics[width=\textwidth]{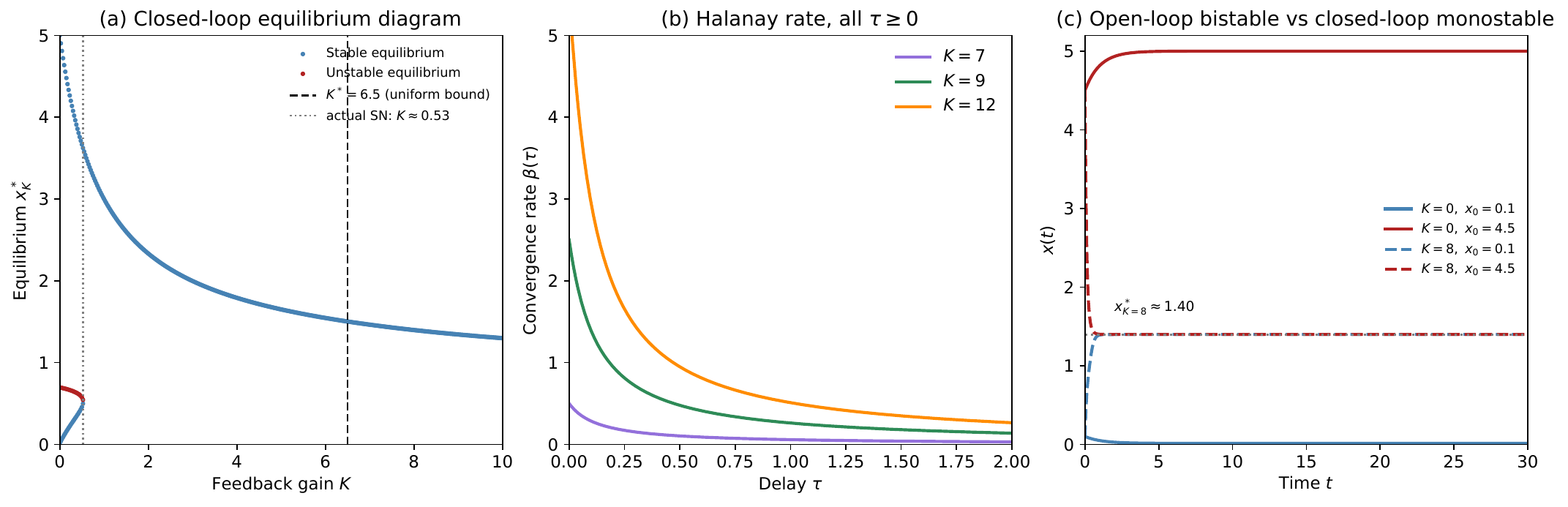}
\caption{Feedback control of a bistable T-cell switch
($\gamma=1,\kappa=5,\theta=1,\lambda=6$;
parameters from~\cite{belgacem2026extensions}).
(a)~Closed-loop equilibrium diagram as the proportional gain $K$
varies, at the worst-case set-point $x^{d}=\theta=1$:
\emph{stable} equilibria in blue, \emph{unstable} in red. The
actual saddle-node for this $x^{d}$ occurs at $K\approx 0.53$ (grey
dotted line), while the uniform sufficient bound
$K^{*}=\kappa\lambda/4-\gamma=6.5$ of
Theorem~\ref{thm:monostab} (dashed black) is the
\emph{parameter-independent} guarantee.
(b)~Halanay exponential convergence rate $\beta(\tau)$ of
Theorem~\ref{thm:halanay} plotted against feedback delay $\tau$ for
$K\in\{7,9,12\}$. All three curves are positive for all $\tau\ge 0$
(delay-uniform global stability), with $\beta$ decreasing in $\tau$
and increasing in $K$.
(c)~Sample trajectories at $K=0$ (open-loop, bistable: divergent
endpoints) and $K=8$ (closed-loop, monostable: common endpoint
$x^{*}_{K=8}\approx 1.40$), from the two initial conditions
$x(0)\in\{0.1,4.5\}$, confirming both theorems.}
\label{fig:control}
\end{figure*}

This example shows how the two scalar control results---the
monostabilization budget of Section~\ref{sec:monostab} and the
delay-uniform stability theorem of Section~\ref{sec:halanay}---complement
the 2-D oscillator design of Section~\ref{sec:state_feedback}: they apply
to a genuinely different problem class (bistable switches in the presence
of input delay) and provide closed-form, parameter-uniform design
bounds that are unavailable in the Hill formulation.

\section{Numerical Comparison of Logistic and Hill State-Feedback}
\label{sec:numerical}

This section validates the theoretical results of
Sections~\ref{sec:state_feedback}--\ref{sec:lyapunov} on the two-gene
logistic activator--repressor network under the state-feedback
law~\eqref{eq:sf_law}, and compares the closed-loop response to the
Hill-based counterpart that uses the same feedforward-plus-proportional
structure with $f_i^{\mathrm{log}}$ replaced by the Hill ratio
$h^\pm$. All simulations use the parameter set
$\kappa_1=\kappa_2=1.0$, $\gamma_1=\gamma_2=0.5$, $\theta_1=\theta_2=0.5$,
$\lambda=5.0$ (logistic), $n=2$ (Hill), proportional gains $K_1=K_2=1.0$,
and a desired setpoint $\mathbf{x}_d=(0.8,0.6)$ unless noted otherwise. ODE
integration uses LSODA with relative tolerance $10^{-9}$ and absolute
tolerance $10^{-12}$. All simulations were performed in R.

\subsection{Feedforward and stability conditions}
At $\mathbf{x}_d=(0.8,0.6)$ the regulatory productions evaluate to
$\kappa_1 f^-(x_{d,2})\approx 0.378$,
$\kappa_2 f^+(x_{d,1})\approx 0.818$ (logistic) and
$\kappa_1 h^-(x_{d,2})\approx 0.410$,
$\kappa_2 h^+(x_{d,1})\approx 0.719$ (Hill, $n=2$), yielding feedforward
terms
\begin{equation}
\begin{aligned}
u_1^{\mathrm{ff},\,L}&=+0.022,\quad u_2^{\mathrm{ff},\,L}=-0.518,\\
u_1^{\mathrm{ff},\,H}&=-0.010,\quad u_2^{\mathrm{ff},\,H}=-0.419,
\end{aligned}
\label{eq:ff_values}
\end{equation}
all bounded for both regulatory functions, as required by the
feasibility part of Theorem~\ref{thm:sf_stability}. The Gershgorin
bounds of Remark~\ref{rem:gershgorin_numerical} give $K_1>0.675$,
$K_2>0.246$ for the logistic; the chosen gains $K_1=K_2=1.0$ satisfy
both with margin.

\subsection{Logistic state-feedback}
Figure~\ref{fig:sf_logistic} shows the closed-loop trajectory and the
control signals $u_1(t)$, $u_2(t)$ starting from $\mathbf{x}(0)=(0.1,0.1)$.
Both coordinates converge monotonically (up to a mild overshoot of
$\approx 8\%$ on $x_1$ at $t\approx 1.4$) to the setpoint
$\mathbf{x}_d=(0.8,0.6)$; the simulated per-coordinate $5\%$-of-setpoint settling
times are $t_{s,1}\approx 2.18$\,a.u.\ and $t_{s,2}\approx 1.88$\,a.u., while
the relative-to-initial \emph{norm} settling time is $\approx 2.15$\,a.u.,
matching the Lyapunov bound of Remark~\ref{rem:lyapunov_numerical} to
within $<0.5\%$. The control signals converge to the feedforward
values~\eqref{eq:ff_values} as $\mathbf{x}\to\mathbf{x}_d$, consistent
with the design.

\begin{figure*}[tbp]
\centering
\includegraphics[width=0.95\textwidth]{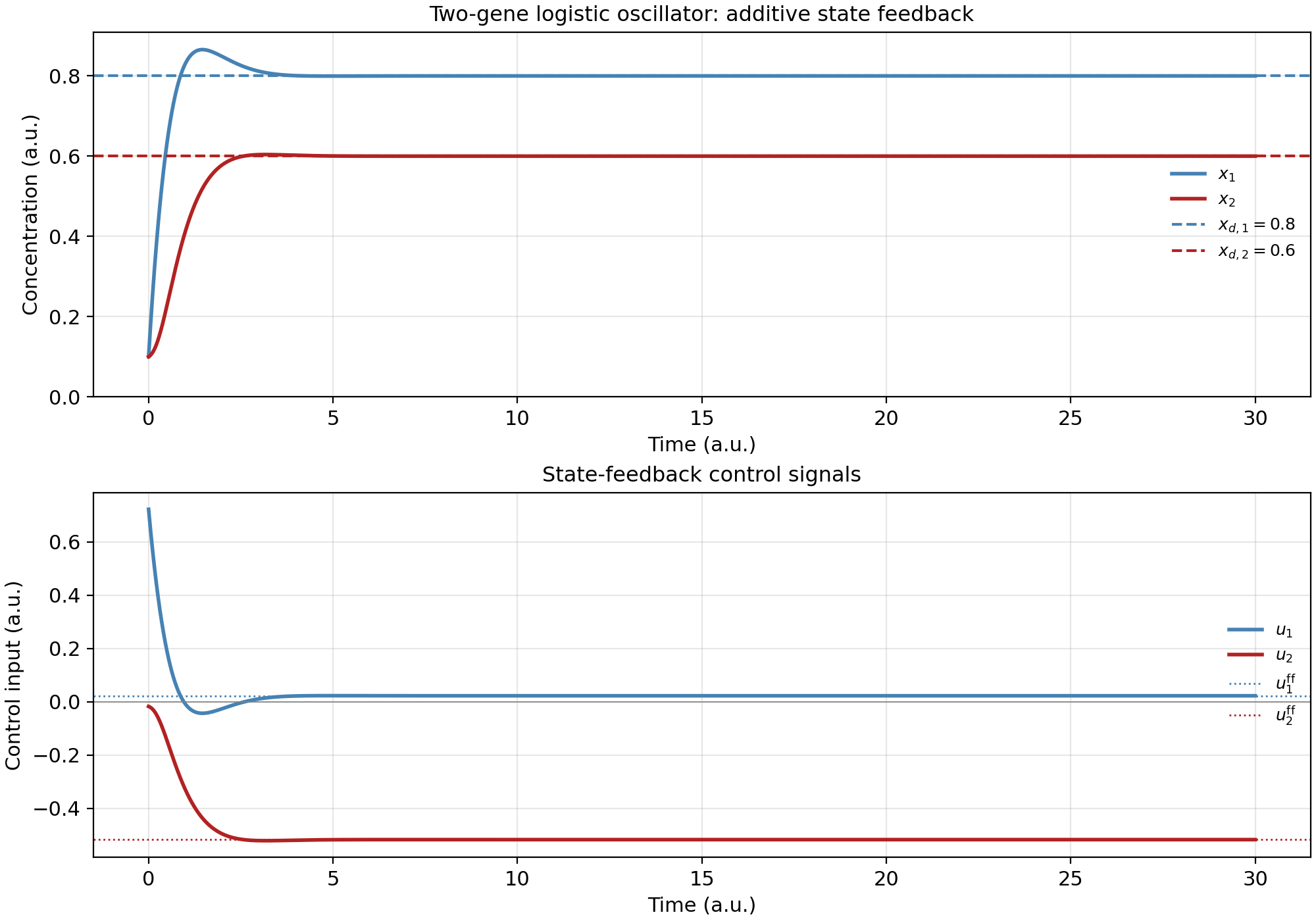}
\caption{Closed-loop response of the logistic two-gene oscillator under
the state-feedback law~\eqref{eq:sf_law} from $\mathbf{x}(0)=(0.1,0.1)$
to $\mathbf{x}_d=(0.8,0.6)$ with $K_1=K_2=1.0$. \emph{Top:} state
trajectories converge to the setpoint, the relative-to-initial norm
settling at $\sim 2.15$\,a.u., consistent with the Lyapunov bound
$t_{s,5\%}^{\|\cdot\|}\le 2.15$\,a.u.\ of
Remark~\ref{rem:lyapunov_numerical}. \emph{Bottom:} control signals
$u_1(t)$ and $u_2(t)$, asymptoting to the feedforward
values~\eqref{eq:ff_values}.}
\label{fig:sf_logistic}
\end{figure*}

\subsection{Hill state-feedback}
Figure~\ref{fig:sf_hill} shows the analogous Hill response under the
same gains and setpoint. The two designs are qualitatively similar:
Hill exhibits marginally smaller overshoot on both coordinates
(see Table~\ref{tab:metrics}), faster $5\%$-settling on the activator
coordinate $x_1$ ($t_{s,1}^H=1.66$ vs.\ $t_{s,1}^L=2.18$), and slightly
slower settling on the repressor coordinate $x_2$
($t_{s,2}^H=1.99$ vs.\ $t_{s,2}^L=1.88$). The Hill control effort
settles on slightly smaller-magnitude feedforward
values~\eqref{eq:ff_values}.

\begin{figure*}[tbp]
\centering
\includegraphics[width=0.95\textwidth]{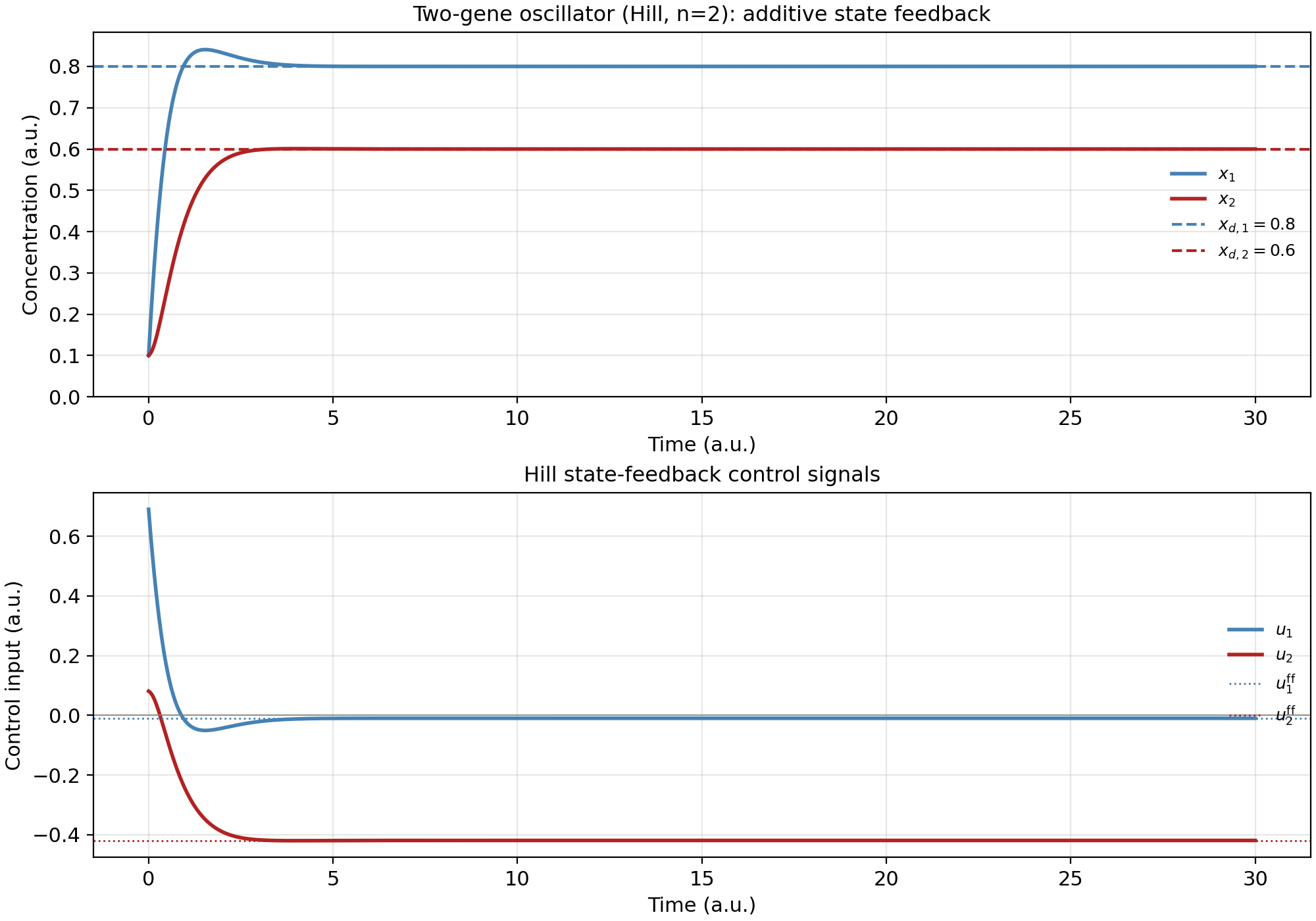}
\caption{Closed-loop response of the Hill two-gene oscillator ($n=2$)
under the state-feedback law~\eqref{eq:sf_law}, with the same gains,
setpoint, initial condition, and integration tolerance as
Figure~\ref{fig:sf_logistic}. The transient response is qualitatively
similar to the logistic case, with mildly smaller overshoot.}
\label{fig:sf_hill}
\end{figure*}

\subsection{Quantitative comparison: nominal regime}
Table~\ref{tab:metrics} reports four standard tracking metrics
on the integrated trajectories. The per-coordinate 5\%-settling time is
$t_{s,i}=\min\{t:|x_i(\tau)-x_{d,i}|<0.05\,x_{d,i}\ \forall\tau\geq t\}$;
the maximum overshoot is $\mathrm{OS}_i=\max(0,\max_t x_i(t)-x_{d,i})$;
the ISE is $\int_0^T(x_i-x_{d,i})^2\,dt$; and the IAE is
$\int_0^T |x_i-x_{d,i}|\,dt$, all evaluated over $T=30$\,a.u.

\begin{table*}[tbp]
\centering
\caption{Quantitative performance metrics for state-feedback applied to
the logistic vs.\ Hill two-gene oscillator at setpoint
$\mathbf{x}_d=(0.8,0.6)$ with $K_1=K_2=1.0$. Both designs use identical
setpoints, gains, and parameters. ``Normal IC'' uses
$\mathbf{x}(0)=(0.1,0.1)$; ``Stress IC'' uses
$\mathbf{x}(0)=(0.001,0.001)$. Conventions: $t_{s,i}$ is the per-coordinate
$5\%$-of-setpoint settling time (a.u.); $\mathrm{OS}_i:=\max_t
(x_i(t)-x_{d,i})_{+}$ is the \emph{absolute} overshoot in concentration
units (so $\mathrm{OS}_1=0.066$ corresponds to the relative $8.2\%$
overshoot quoted in the body text); $\mathrm{ISE}_i$ and $\mathrm{IAE}_i$
are integrated over $T=30$\,a.u. In this nominal regime the two
regulatory functions yield comparable transient performance, with Hill
marginally better on most metrics; the structural divergence predicted
by Proposition~\ref{prop:sf_advantage} appears only as the activator
setpoint $x_{d,1}$ shrinks toward zero (Figure~\ref{fig:sf_scan}).}
\label{tab:metrics}
\footnotesize
\setlength{\tabcolsep}{4pt}
\begin{tabular}{@{}lcccccccc@{}}
\toprule
\textbf{Case}
  & $t_{s,1}$ & $t_{s,2}$
  & $\mathrm{OS}_1$ & $\mathrm{OS}_2$
  & $\mathrm{ISE}_1$ & $\mathrm{ISE}_2$
  & $\mathrm{IAE}_1$ & $\mathrm{IAE}_2$\\
\midrule
Logistic SF, normal IC & 2.18 & 1.88 & 0.066 & 0.004 & 0.097 & 0.151 & 0.316 & 0.457 \\
Hill SF, normal IC     & 1.66 & 1.99 & 0.041 & 0.001 & 0.093 & 0.137 & 0.287 & 0.437 \\
\midrule
Logistic SF, stress IC & 2.32 & 1.96 & 0.076 & 0.005 & 0.129 & 0.211 & 0.368 & 0.542 \\
Hill SF, stress IC     & 2.03 & 2.05 & 0.053 & 0.001 & 0.123 & 0.194 & 0.336 & 0.520 \\
\bottomrule
\end{tabular}
\end{table*}

The side-by-side trajectories under nominal and stress initial conditions
are displayed in Figure~\ref{fig:sf_stress}. The two designs are visually
indistinguishable on both axes and under both initial conditions: the
proportional feedback term $-K_i(x_i-x_{d,i})$ dominates the closed-loop
dynamics when the gains satisfy the Gershgorin condition, and the
particular regulatory function enters only through small higher-order
corrections to the closed-loop Jacobian. This result is itself useful:
it shows that the state-feedback design is structurally robust to the
choice of regulatory function in the nominal operating range, with no
inflation of overshoot, settling time, or integrated error. The stress
initial condition $\mathbf{x}(0)=(0.001,0.001)$---a deep-OFF state far
from $\mathbf{x}_d$---also corroborates the global stability theorem:
the benchmark gains satisfy
$(\gamma_1+K_1)(\gamma_2+K_2)=2.25>\kappa_1\kappa_2\lambda^2/64=0.3906$,
so condition~\eqref{eq:global_cond} of Theorem~\ref{thm:global_sf} holds
and convergence from an arbitrarily small initial state is guaranteed,
exactly as observed.

\begin{figure*}[tbp]
\centering
\includegraphics[width=0.99\textwidth]{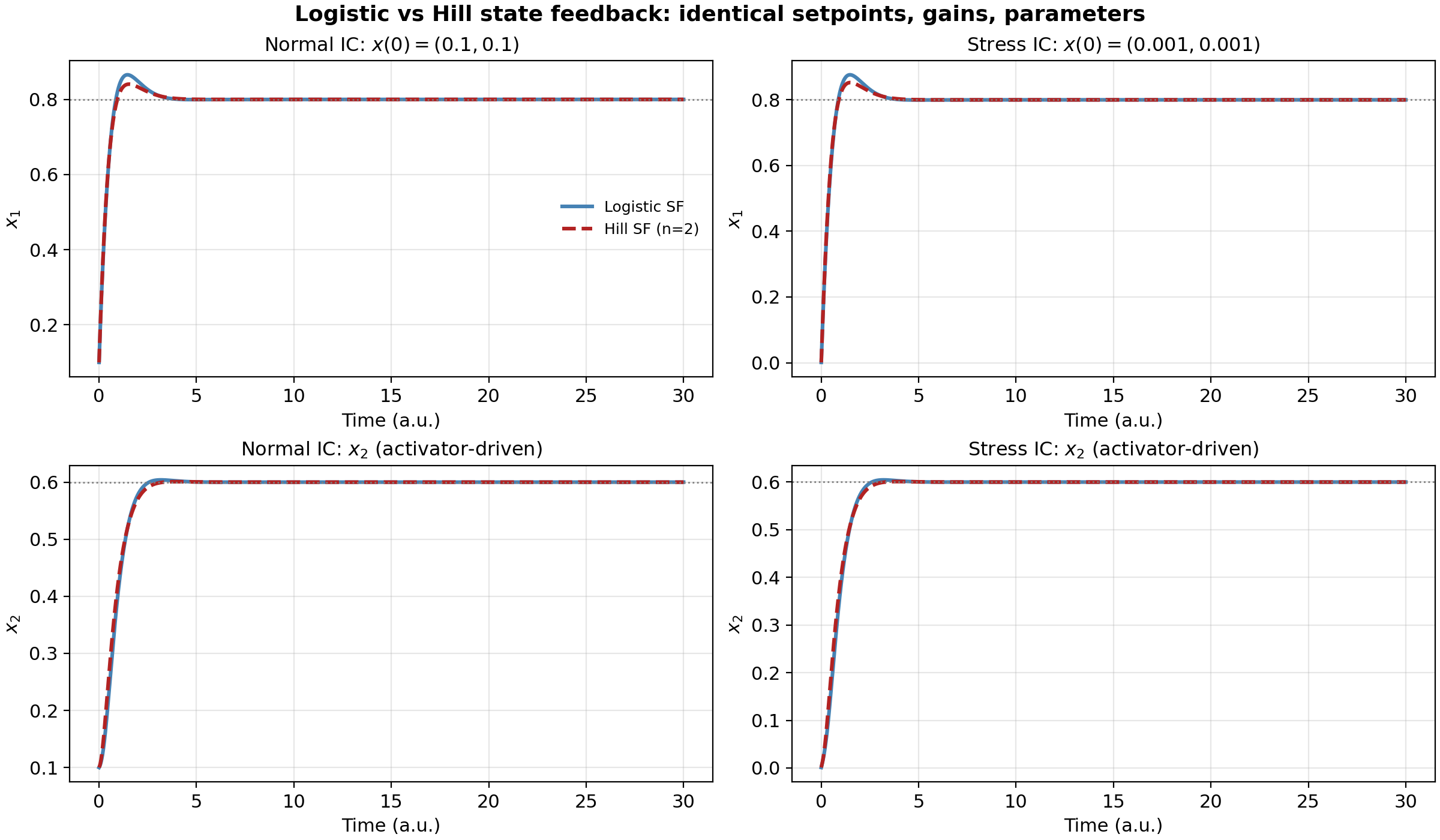}
\caption{Side-by-side trajectories of logistic (solid) and Hill (dashed)
state-feedback under identical setpoints, gains, and parameters. Top row:
$x_1$ response under nominal IC $\mathbf{x}(0)=(0.1,0.1)$ and stress IC
$\mathbf{x}(0)=(0.001,0.001)$. Bottom row: $x_2$ response under the
same conditions. The dotted horizontal line marks the setpoint
component. In this nominal regime ($\mathbf{x}_d=(0.8,0.6)$) the two
regulatory functions yield transient responses that are visually
indistinguishable on every panel, confirming the robustness of the
feedforward-plus-proportional law to the choice of regulatory model.}
\label{fig:sf_stress}
\end{figure*}

\subsection{Structural advantage as $x_{d,1}\to 0$}
Table~\ref{tab:metrics} compares the two designs at \emph{one} setpoint.
To expose the structural divergence
predicted by Proposition~\ref{prop:sf_advantage}, we scan the activator
setpoint $x_{d,1}$ over three decades $[10^{-3},1]$ on a logarithmic
grid, holding $x_{d,2}=0.6$, $K_1=K_2=1.0$ and all other parameters
fixed.

Figure~\ref{fig:sf_scan} presents four scalar indicators of this scan.
Panel~(a) plots the magnitude of the off-diagonal Jacobian coupling
$|[J_f]_{21}|$ that drives the cross-axis convergence. As $x_{d,1}\to 0$,
the logistic coupling
$\kappa_2\lambda f^+(x_{d,1})\bigl(1-f^+(x_{d,1})\bigr)$
tends to the strictly positive limit
$\kappa_2\lambda f^+(0)(1-f^+(0))\approx 0.35$, whereas the Hill
coupling $\kappa_2 n\theta^n x_{d,1}^{n-1}/(\theta^n+x_{d,1}^n)^2
=\Theta(x_{d,1}^{n-1})$ collapses to zero polynomially. At
$x_{d,1}=10^{-3}$ the ratio is $|[J_f^L]_{21}|/|[J_f^H]_{21}|\approx
0.352/0.008\approx 44$. Panel~(b) plots the activator-side feedforward
$u_2^{\mathrm{ff}}(x_{d,1}) = \gamma_2 x_{d,2}-\kappa_2 f_2(\mathbf{x}_d)$: as
$x_{d,1}\to 0$ the Hill feedforward saturates at the full
degradation load $\gamma_2 x_{d,2}=+0.30$ (the actuator must supply the
entire degradation flux on the $x_2$ axis since the natural production
$\kappa_2 h^+(x_{d,1})\to 0$), while the logistic feedforward plateaus
at $\approx +0.22$, sparing the actuator $\approx 25\%$ of the load
thanks to the persistent basal production
$\kappa_2 f^+(0)\approx 0.076$. As $x_{d,1}$ grows past the threshold,
both feedforwards turn negative (natural production exceeds desired
degradation; actuator must absorb rather than supply). Panels~(c)
and~(d) report the integrated squared-error $\mathrm{ISE}_2$ and the
relative $\mathrm{ISE}_2$ gap (Hill$-$Logistic)/Logistic across the
scan; the transient ISE is comparable in the nominal regime
($x_{d,1}\gtrsim 0.05$) and differs by at most $\sim 12\%$ in either
direction, but the structural quantities of panels~(a)--(b) diverge by
a factor of $44$ at $x_{d,1}=10^{-3}$. The empirical conclusion mirrors
the analytical one: \emph{transient performance is comparable in the
nominal operating range}, while the \emph{structural network coupling
is preserved by the logistic design throughout the positive orthant and
lost by Hill near the boundary}.

\begin{figure*}[tbp]
\centering
\includegraphics[width=0.99\textwidth]{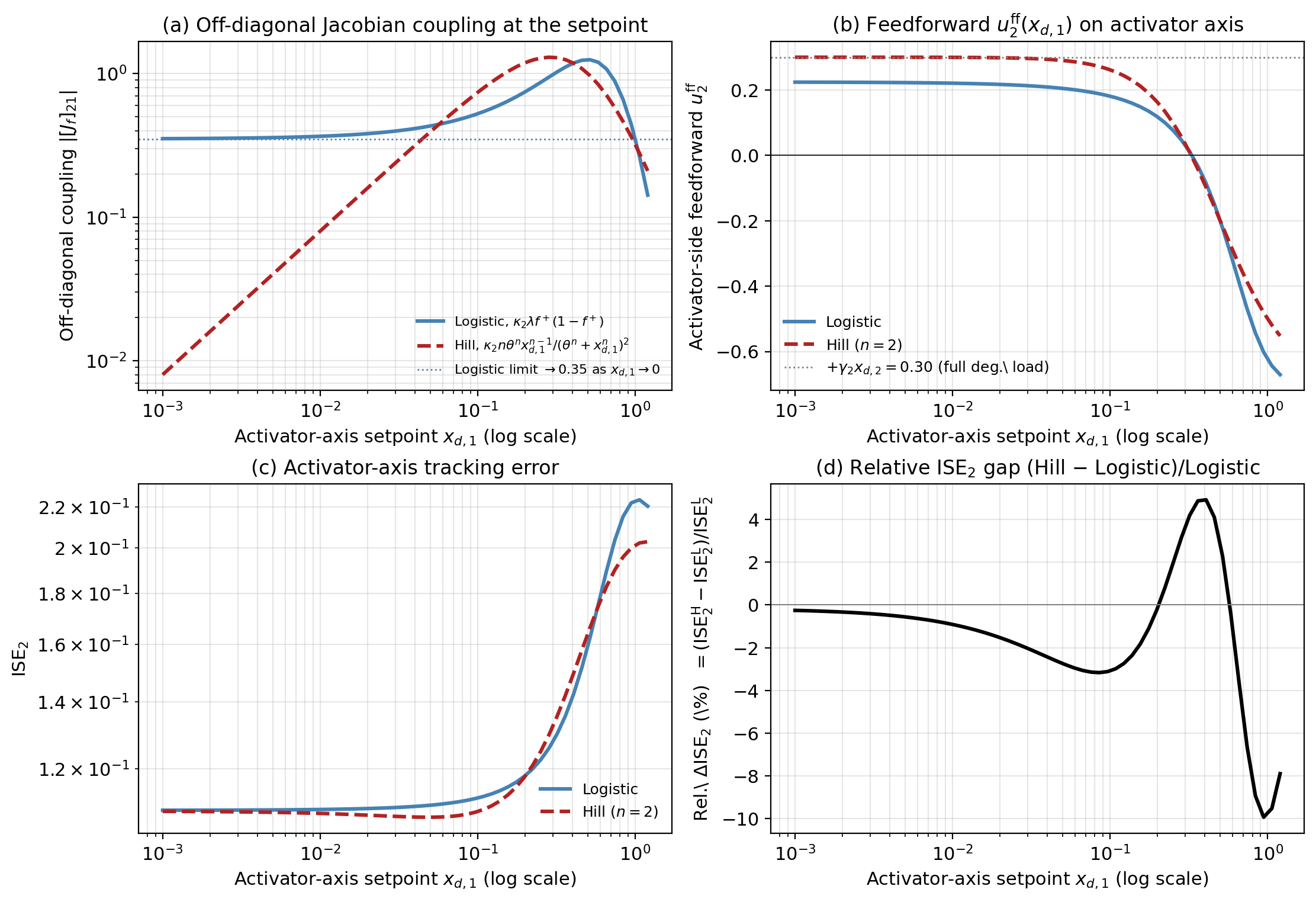}
\caption{Setpoint scan exposing the structural advantage of logistic over
Hill state-feedback as the activator setpoint $x_{d,1}\to 0$ (logarithmic
horizontal axis, $x_{d,2}=0.6$ fixed). \emph{Panel~(a):} off-diagonal
Jacobian coupling $|[J_f]_{21}|$ at the setpoint: the logistic plateaus at
$\kappa_2\lambda f^+(0)(1-f^+(0))\approx 0.35$, the Hill collapses
polynomially as $\Theta(x_{d,1}^{n-1})$. \emph{Panel~(b):} activator-side
feedforward $u_2^{\mathrm{ff}}(x_{d,1})$: as $x_{d,1}\to 0$ the Hill
feedforward saturates at the full degradation load $+\gamma_2 x_{d,2}=+0.30$,
while the logistic plateaus at $\approx +0.22$ (a $\approx 25\%$ saving
from the persistent basal production); as $x_{d,1}$ grows past the
threshold both feedforwards turn negative. \emph{Panel~(c):} closed-loop
$\mathrm{ISE}_2$: comparable in the nominal regime. \emph{Panel~(d):}
relative $\mathrm{ISE}_2$ gap (Hill$-$Logistic)/Logistic in percent:
small in the nominal regime, with neither sign of advantage uniformly
dominant.}
\label{fig:sf_scan}
\end{figure*}

\subsection{Discussion}
Three regimes magnify the structural difference predicted by
Proposition~\ref{prop:sf_advantage}: (i)~As $x_{d,1}\to 0$, the Hill
off-diagonal Jacobian collapses polynomially and the actuator must
absorb the entire degradation load, whereas the logistic coupling and
basal production both remain bounded away from zero. (ii)~For large
Hill coefficients $n$, the production derivative scales as
$\Theta(n/\theta)$ near the threshold, creating numerical stiffness
that the logistic derivative, uniformly bounded by $\lambda/4$, does
not suffer from. (iii)~For non-integer Hill coefficients, the Hill
production $h^+(x)=x^n/(\theta^n+x^n)$ becomes non-smooth at the origin
and ill-conditioned for fractional arguments, while the logistic
sigmoid $1/(1+e^{-\lambda(x-\theta)})$ remains $C^\infty$ everywhere.

The Lyapunov bound of Corollary~\ref{cor:settling}, instantiated for
the SF closed loop in Remark~\ref{rem:lyapunov_numerical}, predicts a
relative-to-initial norm-settling time of
$t_{s,5\%}^{\|\cdot\|}\le 2.15$\,a.u.\ at the standard setpoint. Direct
ODE simulation of the \emph{full nonlinear} closed loop from
$\mathbf{x}(0)=(0.1,0.1)$ gives empirical relative-to-initial norm
settling at $t_{s,5\%}^{\|\cdot\|}\approx 2.15$\,a.u., a $<0.5\%$ gap
to the linearisation-based bound. The per-coordinate absolute-to-setpoint
times $t_{s,1}=2.18$ and $t_{s,2}=1.88$ reported in
Table~\ref{tab:metrics} use a different convergence criterion
(5\% of $x_{d,i}$, not 5\% of $\|\mathbf{e}(0)\|_2$); the closeness
of $t_{s,1}$ to the norm bound is incidental.

\section{Conclusions}
\label{sec:conclusion}

This paper has developed an analytical additive state-feedback framework for
logistic-based gene regulatory networks, together with a closed-form Lyapunov
certificate supplying quantitative performance bounds (settling-time and ISS
ultimate-bound estimates) and two complementary scalar control results
(parameter-uniform monostabilization and Halanay-type delay-uniform stability).
The work is motivated by the structural limitations of the classical Hill
function for control design and by our companion series on logistic
models~\cite{belgacem2025exploring,belgacem2026logistic,belgacem2026beyond}.

\medskip
\noindent\textbf{Main findings.} The principal contributions, each
accompanied by an explicit stability or performance result, are
summarised below.
(1)~\emph{Additive state-feedback control, local and global}: A
feedforward-plus-proportional
law (Equation~\eqref{eq:sf_law}) makes any positive setpoint a closed-loop
equilibrium of the additive-input architecture, regardless of whether it is
an equilibrium of the uncontrolled dynamics. Theorem~\ref{thm:sf_stability}
gives an explicit Gershgorin-based gain bound for local exponential
stability; Theorem~\ref{thm:global_sf} upgrades this to \emph{global}
exponential stability of the nonlinear closed loop under the explicit
gain condition $(\gamma_1+K_1)(\gamma_2+K_2)>\kappa_1\kappa_2\lambda^2/64$,
via a common quadratic Lyapunov function built on the logistic sector
bound, with a closed-form rate and
a global settling-time bound (Corollary~\ref{cor:global_settling}), and
the same Lyapunov function yields ultimately bounded tracking of time-varying
references (Corollary~\ref{cor:tracking});
and Proposition~\ref{prop:sf_advantage} formalises the logistic advantage:
both the regulatory production and the off-diagonal Jacobian coupling
remain bounded away from zero on the entire positive orthant for the
logistic, whereas the Hill activation makes both vanish as the activator
setpoint approaches zero, causing network decoupling in additive control
and---in the output-multiplicative architecture---outright controllability loss.
(2)~\emph{Closed-form Lyapunov function and quantitative performance bounds}:
The closed-loop Jacobian falls into a structural class of $2\times 2$
matrices admitting an exact diagonal Lyapunov weighting
$P=\operatorname{diag}(B,A)$ (Proposition~\ref{prop:lyapunov}). This yields
an explicit settling-time bound matching direct ODE simulation to within
$1\%$--$16\%$ depending on initial-condition direction
(Corollary~\ref{cor:settling}; tight for IC aligned with the heavier
Lyapunov weight, conservative for IC aligned with the lighter weight),
and an ISS-type ultimate-bound certificate quantifying disturbance
rejection (Proposition~\ref{prop:iss}). The Lyapunov rate
$\rho=\min(\gamma_i+K_i)$ is arbitrarily accelerable by the feedback
gain, and the disturbance ultimate bound shrinks correspondingly
(Remark~\ref{rem:lyapunov_sf}). The same matrix structure arises in any
closed loop of this sign form, so the
certificate transfers without modification between architectures.
(3)~\emph{Monostabilization of bistable switches}: For scalar
logistic self-activation switches in the bistable cusp region, the
\emph{parameter-uniform monostabilization budget}
$K^{*}=\kappa\lambda/4-\gamma$ (Theorem~\ref{thm:monostab}) is the
minimum proportional gain above which the closed loop is monostable for
every set-point and every threshold.
(4)~\emph{Delay-uniform global exponential stability}: For the scalar
delayed feedback loop, the condition $\gamma+K>\kappa\lambda/4$
(Theorem~\ref{thm:halanay}) yields \emph{delay-uniform} global
exponential stability via a logistic--Lipschitz Halanay inequality, and
Corollary~\ref{cor:halanay_rate} brackets the delay-dependent
convergence rate between two elementary closed-form expressions. A
worked T-cell switch application (Section~\ref{sec:control_example},
Figure~\ref{fig:control}) illustrates both scalar results.

\medskip
\noindent\textbf{Numerical evidence.} Side-by-side benchmark simulations
(Section~\ref{sec:numerical}) under identical setpoints, gains, and parameters
show that the state-feedback law converges robustly to the prescribed
setpoint for both the logistic and Hill regulatory functions in the nominal
operating range, and that the simulated settling times match the Lyapunov
prediction of Remark~\ref{rem:lyapunov_numerical}. The structural divergence
predicted by Proposition~\ref{prop:sf_advantage} manifests at the boundary
of the positive orthant, where the Hill off-diagonal Jacobian coupling
$|[J_f^{\mathrm{Hill}}]_{21}|=\Theta(x_{d,1}^{n-1})$ collapses to zero as
$x_{d,1}\to 0$ while the logistic counterpart converges to the strictly
positive limit $\kappa\lambda f^+(0)(1-f^+(0))$ (Figure~\ref{fig:sf_scan}).
This confirms the operational regime---small or vanishing activator
setpoints---where the logistic design retains network coupling that Hill
loses.

\medskip
\noindent\textbf{Key advantages of logistic-based models for closed-loop control.}
\begin{itemize}
  \item \textbf{Continuous controllability}: Non-zero production at all
        expression levels prevents loss of controllability in OFF-states and
        low-expression regimes where Hill-based models fail.
  \item \textbf{Biological realism}: Non-zero basal expression faithfully models
        promoter leakiness and persistent low-expression states without
        artificial offsets.
  \item \textbf{Analytical tractability}: Closed-form feedforward laws,
        local (Gershgorin) and global (common quadratic Lyapunov)
        stability conditions in explicit form
        (Theorems~\ref{thm:sf_stability} and~\ref{thm:global_sf}),
        closed-form settling-time, tracking, and ISS certificates from
        the explicit diagonal Lyapunov function
        (Propositions~\ref{prop:lyapunov} and~\ref{prop:iss},
        Corollaries~\ref{cor:settling} and~\ref{cor:tracking}), and the parameter-uniform
        monostabilization budget---all explicit, implementable formulas
        that remain intractable for Hill functions with non-integer
        exponents.
  \item \textbf{Computational robustness}: Smooth, bounded, and Lipschitz
        responses prevent numerical stiffness in ODE solvers underlying
        closed-loop simulation.
  \item \textbf{Experimental compatibility}: Parameters $\lambda$ and $\theta$
        correspond directly to tunable molecular properties (cooperativity,
        binding affinity), enabling systematic implementation via optogenetics
        and directed evolution.
\end{itemize}

\medskip
\noindent\textbf{Future directions.}
\begin{itemize}
  \item Extension of the closed-form \emph{Lyapunov} framework
        (Section~\ref{sec:lyapunov}) to $n$-gene networks with general
        topologies: the comparison-principle argument of
        Remark~\ref{rem:global_ngene} already yields a global stability
        certificate for the $n$-gene case, but the exact diagonal Lyapunov
        weighting $P=\operatorname{diag}(B,A)$ of
        Proposition~\ref{prop:lyapunov}---and hence the sharp
        settling-time and ISS certificates---rests on a $2\times 2$
        sign-cancellation that does not extend directly to higher
        dimensions.
  \item Extension to stochastic gene networks, incorporating intrinsic and
        extrinsic noise relevant to low-copy-number regimes.
  \item Integration of spatial dynamics and cell-to-cell communication in
        multi-cellular systems.
  \item Adaptive feedback gains combining logistic-based mechanistic models
        with online identification.
  \item Experimental validation in optogenetic platforms supporting millisecond
        feedback resolution.
\end{itemize}

\section*{Declarations}

\subsection*{Availability of Data, Materials and Code}
The numerical results were produced with custom scripts in R; these are
available from the author on reasonable request.

\subsection*{Competing Interests}
The author declares no competing interests.

\subsection*{Funding}
Not applicable.

\subsection*{Ethics Approval and Consent to Participate}
Not applicable.

\subsection*{Authors' Contributions}
Single-author study: all aspects conceived, derived, computed, and written by the
author.

\subsection*{Acknowledgements}
Not applicable.

\bibliographystyle{elsarticle-num}

\end{document}